\documentclass[a4paper,11pt,reqno]{amsart}

\usepackage[margin=1in,includehead]{geometry}
\usepackage[utf8]{inputenc}
\usepackage[T1]{fontenc}
\usepackage{lmodern,microtype,mathtools,amsthm,amssymb,enumitem,thmtools,xcolor,tikz}
\usepackage[foot]{amsaddr}
\usepackage{hyperref}
\usepackage[nameinlink]{cleveref}
\usepackage[absolute]{textpos}

\hypersetup{
  pdfauthor={Gwenaël Joret, Piotr Micek, Michał Pilipczuk, Bartosz Walczak},
  pdftitle={Cliquewidth and dimension},
  colorlinks,
  linkcolor={red!50!black},
  citecolor={blue!50!black},
  urlcolor={blue!80!black}
}

\setlist[itemize]{leftmargin=1.5em, labelsep=8pt}
\setlist[enumerate]{label=\upshape{(\arabic*)}, leftmargin=*}

\newtheorem{theorem}{Theorem}
\newtheorem*{colcombet}{Colcombet's Theorem}
\newtheorem{lemma}[theorem]{Lemma}
\newtheorem{fact}[theorem]{Fact}
\theoremstyle{remark}
\newtheorem{claim}[theorem]{Claim}
\crefname{claim}{claim}{claims}
\newtheorem{case}{Case}
\newtheorem{subcase}{Case}[case]

\newenvironment{question}{\addvspace{\medskipamount}\begin{narrower}\itshape\noindent\ignorespaces}{\par\end{narrower}\addvspace{\medskipamount}}

\newcommand\setN{\mathbb{N}}
\newcommand\cgB{\mathcal{B}}
\newcommand\cgC{\mathcal{C}}
\newcommand\cgF{\mathcal{F}}
\newcommand\cgM{\mathcal{M}}
\newcommand\cgP{\mathcal{P}}
\newcommand\Oh{\mathcal{O}}
\newcommand\decode{\mathsf{decode}}
\newcommand\se{\operatorname{se}}
\newcommand\kelly{\operatorname{Kelly}}
\newcommand\id{\mathrm{id}}
\newcommand\inner{{\mathrm{inner}}}
\renewcommand\outer{{\mathrm{outer}}}
\newcommand\downto[1][T]{\prec_{#1}}
\newcommand\downtoeq[1][T]{\preceq_{#1}}

\newcommand\lablr[2]{L^\leq(#1,#2)}
\newcommand\labgr[2]{L^\geq(#1,#2)}
\newcommand\bagll[2]{\eta^\leq(#1,#2)}
\newcommand\baglb[2]{\rho^\leq(#1,#2)}
\newcommand\baggl[2]{\eta^\geq(#1,#2)}
\newcommand\baggb[2]{\rho^\geq(#1,#2)}
\newcommand\bagbl[2]{\eta(#1,#2)}
\newcommand\bagbb[2]{\rho(#1,#2)}
\newcommand\initial[1]{\eta(#1)}

\newcommand\refpreceq{\preceq_*}
\newcommand\dir{\mathsf{dir}}
\newcommand\agg{\mathsf{agg}}
\newcommand\rel{\mathsf{rel}}
\newcommand\init{\mathsf{init}}

\DeclareMathOperator\bdim{bdim}

\DeclarePairedDelimiter\set{\{}{\}}
\DeclarePairedDelimiter\size{\lvert}{\rvert}

\let\leq\leqslant
\let\geq\geqslant
\let\nleq\nleqslant

\let\preceq\preccurlyeq
\let\succeq\succcurlyeq

\let\from\leftarrow
\let\setminus-
\let\Gamma\varGamma
\let\Lambda\varLambda
\let\Phi\varPhi
\let\Psi\varPsi
\let\rho\varrho
\let\epsilon\varepsilon

\let\type\Diamond

\begin{document}

\title{Cliquewidth and dimension}

\author[Joret]{Gwenaël Joret}
\email{gwenael.joret@ulb.be}
\author[Micek]{Piotr Micek}
\email{piotr.micek@uj.edu.pl}
\author[Pilipczuk]{Michał Pilipczuk}
\email{michal.pilipczuk@mimuw.edu.pl}
\author[Walczak]{Bartosz Walczak}
\email{bartosz.walczak@uj.edu.pl}

\address[G.~Joret]{Computer Science Department, Université libre de Bruxelles, Brussels, Belgium}
\address[P.~Micek, B.~Walczak]{Theoretical Computer Science Department, Faculty of Mathematics and Computer Science, Jagiellonian University, Kraków, Poland}
\address[Mi.~Pilipczuk]{Institute of Informatics, Faculty of Mathematics, Informatics, and Mechanics, University of Warsaw, Poland}

\thanks{G.~Joret is supported by a PDR grant from the Belgian National Fund for Scientific Research (FNRS).
P.~Micek is supported by the National Science Center of Poland under grant UMO-2022/47/B/ST6/02837 within the OPUS 24 program.
Mi.~Pilipczuk is supported by the project BOBR that has received funding from the European Research Council (ERC) under the European Union's Horizon 2020 research and innovation programme (grant agreement No 948057).
B.~Walczak is partially supported by the National Science Center of Poland grant 2019/34/E/ST6/00443.}

\begin{textblock}{20}(0.4, 13.45)
\includegraphics[width=35px]{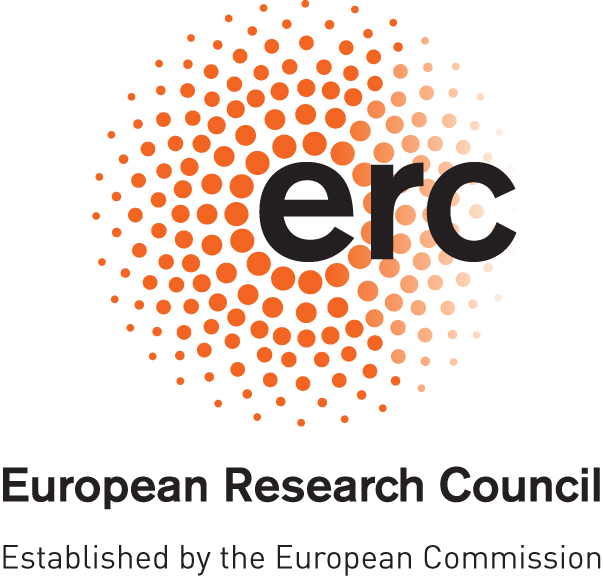}%
\end{textblock}
\begin{textblock}{20}(0.4, 14.25)
\includegraphics[width=35px]{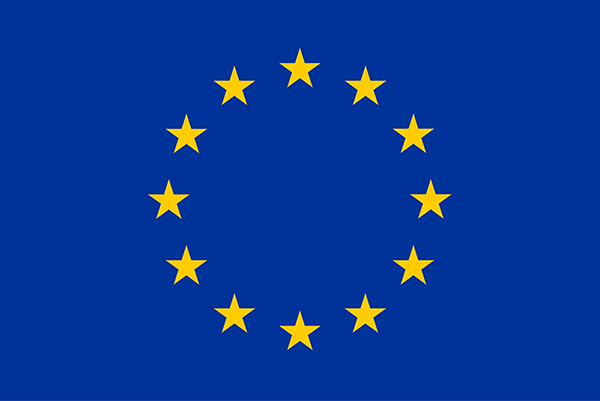}%
\end{textblock}

\begin{abstract}
We prove that every poset with bounded cliquewidth and with sufficiently large dimension contains the standard example of dimension $k$ as a subposet.
This applies in particular to posets whose cover graphs have bounded treewidth, as the cliquewidth of a poset is bounded in terms of the treewidth of the cover graph.
For the latter posets, we prove a stronger statement: every such poset with sufficiently large dimension contains the Kelly example of dimension $k$ as a subposet.
Using this result, we obtain a full characterization of the minor-closed graph classes $\mathcal{C}$ such that posets with cover graphs in $\mathcal{C}$ have bounded dimension: they are exactly the classes excluding the cover graph of some Kelly example.
Finally, we consider a variant of poset dimension called Boolean dimension, and we prove that posets with bounded cliquewidth have bounded Boolean dimension.

The proofs rely on Colcombet's deterministic version of Simon's factorization theorem, which is a fundamental tool in formal language and automata theory, and which we believe deserves a wider recognition in structural and algorithmic graph theory.
\end{abstract}

\maketitle

\section{Introduction}\label{sec:intro}

Dimension is a key measure of complexity of partially ordered sets.%
\footnote{All partially ordered sets considered in this paper are finite.}
It is defined as follows.
A~\emph{linear extension} $L$ of a poset $P$ is a linear order on the elements of $P$ such that if $x\leq y$ in $P$, then $x\leq y$ in $L$.
A \emph{realizer} of a poset $P$ is a set $\set{L_1,\ldots,L_k}$ of
linear extensions of $P$ such that
\[x\leq y\text{ in }P\quad\text{if and only if}\quad (x\leq y\text{ in }L_1){,}\enspace\ldots{,}\enspace(x\leq y\text{ in }L_k){,}\]
for all $x,y\in P$.
The \emph{dimension} of $P$, denoted by $\dim(P)$, is the minimum size of a realizer of $P$.

The notion of dimension goes back to the work of Dushnik and Miller~\cite{DM41} in 1941, where they also identified a canonical structure forcing dimension to be large, namely, the standard examples.
For an integer $k\geq 2$, the \emph{standard example} $S_k$ is a poset whose ground set is $\set{a_1,\ldots,a_k,b_1,\ldots,b_k}$ with $a_i<b_j$ in $S_k$ if and only if $i\neq j$; see Figure~\ref{fig:standard-and-Kelly}.
They observed that $\dim(S_k)=k$, and thus every poset containing $S_k$ as a subposet has dimension at least $k$.

In many ways, dimension for posets resembles chromatic number for graphs, and standard examples in posets are similar to cliques in graphs.
It is well known that there are triangle-free graphs with arbitrarily large chromatic number.
Similarly, there are posets (specifically, interval orders) which contain no $S_2$ and have arbitrarily large dimension.
In the last few decades, a rich theory of $\chi$-boundedness was developed, whose purpose is to identify classes of graphs such that the chromatic number is bounded from above by a function of the clique number; see e.g.\ a survey by Scott and Seymour~\cite{SS20}.
In this paper, we study the poset equivalent of $\chi$-boundedness, defined as follows.

The \emph{standard example number} of a poset $P$, denoted by $\se(P)$, is set to be $1$ if $P$ does not contain a subposet isomorphic to $S_2$; otherwise $\se(P)$ is the largest $k\geq 2$ such that $P$ contains a subposet isomorphic to $S_k$.
A class $\cgC$ of posets is \emph{$\dim$-bounded} if there is a function $f\colon\setN\to\setN$ such that $\dim(P)\leq f(\se(P))$ for every poset $P$ in $\cgC$.

While the notion of $\dim$-boundedness is natural, and the studies on dimension theory goes back to the 1970s, establishing positive results proved to be surprisingly difficult.
In fact, until recently, no non-trivial class of posets was shown to be $\dim$-bounded, even though many were conjectured to be.
The first contribution of this paper is as follows.

\begin{theorem}\label{thm:dim-bounded-main}
Posets of bounded cliquewidth are\/ $\dim$-bounded.
\end{theorem}

\emph{Cliquewidth} is a graph parameter that measures tree-likeness of a graph; it is the number of labels necessary to build the graph using some simple operations.
It generalizes treewidth, another graph parameter measuring tree-likeness, in the sense that an upper bound on the treewidth of a graph implies an upper bounded on its cliquewidth.
Cliquewidth was first described explicitly by Courcelle and Olariu~\cite{CO00}, though the idea dates back to the work on graph grammars in the 1990s.
In fact, throughout this work we will use an earlier parameter \emph{NLC-width}, introduced by Wanke~\cite{Wanke94}, which is functionally equivalent to cliquewidth in the sense that one can be bounded by a function of the other.
While these parameters were originally defined for undirected graphs, their definitions easily extend to arbitrary binary structures, in particular to directed graphs or posets.
We postpone a formal definition of NLC-width to \Cref{sec:overview,sec:prelims}.

The key ingredient in the proof of \Cref{thm:dim-bounded-main} is the deterministic variant of Simon's Factorization Theorem, due to Colcombet~\cite{Colcombet07}, which is a fundamental tool in the theory of formal languages and automata.
Colcombet's Theorem concerns the setting of a tree $T$ that is edge-labeled by elements of a finite semigroup $(\Lambda,\cdot)$, and it provides a hierarchical decomposition of $T$ (called \emph{factorization}) into smaller and smaller subtrees (called \emph{factors}) satisfying the following: (i) the decomposition has depth depending only on $\size{\Lambda}$, and (ii) every split of a factor into smaller factors present in the decomposition obeys a certain Ramseyan property.
Originally, in~\cite{Colcombet07}, Colcombet used his result to study the expressive power of monadic second-order logic on (infinite) trees.
Applications of Colcombet's Theorem in structural graph theory were pioneered by Bonamy and Pilipczuk~\cite{BP20}, who used a delicate induction on Colcombet factorization to prove that classes of graphs with bounded cliquewidth are polynomially $\chi$-bounded.
Later, Nešetřil, Ossona de Mendez, Pilipczuk, Rabinovich, and Siebertz~\cite{NMPRS21} applied a similar but more involved approach to prove that classes of bounded cliquewidth that exclude a half-graph as a semi-induced subgraph can be transduced in first-order logic from classes of bounded treewidth; this in particular implies that such classes are linearly $\chi$-bounded.

Our usage of Colcombet's Theorem in the proof of \Cref{thm:dim-bounded-main} is very much in the spirit of~\cite{BP20,NMPRS21}.
We view a clique expression (or rather an NLC-decomposition) of a poset as a tree edge-labeled with elements of some very carefully crafted semigroup.
Roughly speaking, these elements encode the relabeling functions present in the expression as well as information on the existence of some basic ``building blocks'' for standard examples, which we call \emph{crosses}.
Then we proceed by induction on Colcombet factorization of the obtained edge-labeled tree to construct a coloring of the incomparable pairs that avoids monochromatic alternating cycles and uses few colors; this is a standard approach to bounding the dimension of a poset.
In each step of the induction, we either succeed in extending the coloring to a larger factor, or we can use the Ramseyan property of the factorization to expose a large standard example.

\vskip 2ex plus 1ex

Let us provide some context for \Cref{thm:dim-bounded-main}.
The \emph{cover graph} of a poset $P$ is the graph on the elements of $P$ with edge set $\{xy\colon x<y$ in $P$ and there is no $z$ with $x<z<y$ in $P\}$.
If we take the lens of structural graph theory, it is natural to wonder whether excluding a fixed minor from the cover graphs leads to posets that are $\dim$-bounded.
This problem is very much open; in fact, a central open problem in the area is to decide whether posets with planar cover graphs are $\dim$-bounded (see~\cite{BMT22} for the current state of the art on this problem).
To attack these questions, it makes sense to first consider some well-understood and rather simple minor-closed classes of graphs, such as graphs of bounded pathwidth or bounded treewidth.
Indeed, our original goal was precisely to show that posets with cover graphs of bounded treewidth are $\dim$-bounded.
As it turned out, the proof works even in the setting of posets of bounded cliquewidth, which is broader---the cliquewidth of a poset can be bounded from above by a function of the treewidth of its cover graph.

Going back to posets with planar cover graphs, it was even speculated for a short time in the 1970s that they had dimension bounded by an absolute constant, which could then be considered a poset analog of the Four-Color Theorem.
However, in 1981, Kelly~\cite{Kelly81} presented a family of posets with planar cover graphs and arbitrarily large dimension.
For an integer $k\geq 3$, the \emph{Kelly example} $\kelly_k$ is a poset whose ground set is $\set{a_2,\ldots,a_{k-1},b_2,\ldots,b_{k-1},c_1,\ldots,c_{k-1},d_1,\ldots,d_{k-1}}$ with $a_2,\ldots,a_i<c_i<b_{i+1},\ldots,b_{k-1}$ and $a_{k-1},\ldots,a_{i+1}<d_i<b_i,\ldots,b_2$ for all $i\in\set{1,\ldots,k-1}$, and with $c_1<\cdots<c_{k-1}$ and $d_1<\cdots<d_{k-1}$; see \Cref{fig:standard-and-Kelly}.%
\footnote{We remark that this is actually a slight variation of the original construction of Kelly.
The latter has four extra elements of degree $1$ in the cover graph, but these are unnecessary.
Our theorems are true for both versions of the construction, we just find it simpler not to have to deal with these vertices of degree $1$ in the proofs.}
In fact, $\kelly_k$ contains $S_k$ as a subposet, so $\dim(\kelly_k)\geq\se(\kelly_k)\geq k$ for all $k\geq 3$.

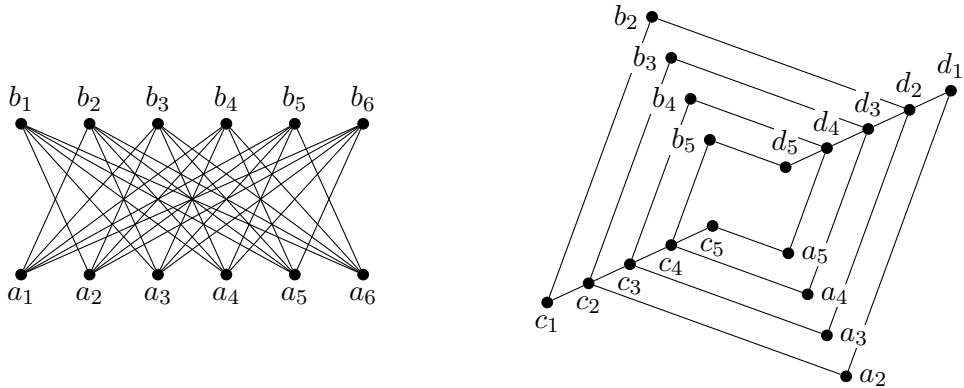
\begin{figure}
\centering
\tikzset{every node/.style={circle,draw,fill,minimum size=4pt,inner sep=0pt}}
\tikzset{every label/.style={rectangle,draw=none,fill=white,inner sep=1pt,label distance=1.5pt}}
\begin{tikzpicture}[xscale=.9,yscale=2,baseline=(current bounding box.center)]
\node[label=below:$a_1$] (a1) at (0,0) {};
\node[label=below:$a_2$] (a2) at (1,0) {};
\node[label=below:$a_3$] (a3) at (2,0) {};
\node[label=below:$a_4$] (a4) at (3,0) {};
\node[label=below:$a_5$] (a5) at (4,0) {};
\node[label=below:$a_6$] (a6) at (5,0) {};
\node[label=above:$b_1$] (b1) at (0,1) {};
\node[label=above:$b_2$] (b2) at (1,1) {};
\node[label=above:$b_3$] (b3) at (2,1) {};
\node[label=above:$b_4$] (b4) at (3,1) {};
\node[label=above:$b_5$] (b5) at (4,1) {};
\node[label=above:$b_6$] (b6) at (5,1) {};
\path (a1) edge (b2) edge (b3) edge (b4) edge (b5) edge (b6);
\path (a2) edge (b1) edge (b3) edge (b4) edge (b5) edge (b6);
\path (a3) edge (b1) edge (b2) edge (b4) edge (b5) edge (b6);
\path (a4) edge (b1) edge (b2) edge (b3) edge (b5) edge (b6);
\path (a5) edge (b1) edge (b2) edge (b3) edge (b4) edge (b6);
\path (a6) edge (b1) edge (b2) edge (b3) edge (b4) edge (b5);
\end{tikzpicture}\hskip 2cm
\begin{tikzpicture}[scale=.6,rotate=25,baseline=(current bounding box.center)]
\node[label=below:$c_1$] (c1) at (-5,-0.25) {};
\node[label=above:$d_1$] (d1) at (5,0.25) {};
\node[label={[yshift=-0.5pt]right:$a_2$}] (a2) at (0.25,-4.5) {};
\node[label=left:$b_2$] (b2) at (-0.25,4.5) {};
\node[label=below:$c_2$] (c2) at (-4,-0.25) {};
\path (c1) edge (b2);
\path (d1) edge (a2);
\node[label=above:$d_2$] (d2) at (4,0.25) {};
\node[label={[yshift=-0.5pt]right:$a_3$}] (a3) at (0.25,-3.5) {};
\node[label=left:$b_3$] (b3) at (-0.25,3.5) {};
\node[label=below:$c_3$] (c3) at (-3,-0.25) {};
\path (c2) edge (c1) edge (a2) edge (b3);
\path (d2) edge (d1) edge (b2) edge (a3);
\node[label=above:$\!d_3\!$] (d3) at (3,0.25) {};
\node[label={[yshift=-0.5pt]right:$a_4$}] (a4) at (0.25,-2.5) {};
\node[label=left:$b_4$] (b4) at (-0.25,2.5) {};
\node[label=below:$c_4$] (c4) at (-2,-0.25) {};
\path (c3) edge (c2) edge (a3) edge (b4);
\path (d3) edge (d2) edge (b3) edge (a4);
\node[label=above:$\!d_4\!$] (d4) at (2,0.25) {};
\node[label={[yshift=-0.5pt]right:$a_5$}] (a5) at (0.25,-1.5) {};
\node[label=left:$b_5$] (b5) at (-0.25,1.5) {};
\node[label=below:$c_5$] (c5) at (-1,-0.25) {};
\path (c4) edge (c3) edge (a4) edge (b5);
\path (d4) edge (d3) edge (b4) edge (a5);
\node[label=above:$\!d_5\!$] (d5) at (1,0.25) {};
\path (c5) edge (c4) edge (a5);
\path (d5) edge (d4) edge (b5);
\end{tikzpicture}
\caption{The standard example $S_6$ (left) and $\kelly_6$ (right).
The poset relation goes upwards in the drawing.
The pairs $(c_1,d_1),(a_2,b_2),\ldots,(a_5,b_5),(d_5,c_5)$ in $\kelly_6$ induce $S_6$ as a subposet.
As we can see, $\kelly_k$ has a planar diagram, and the cover graph of $\kelly_k$ has treewidth (and even pathwidth) at most $3$.}
\label{fig:standard-and-Kelly}
\end{figure}

We show that posets with cover graphs of bounded treewidth and with large dimension contain not only large standard examples but also large Kelly examples.

\begin{theorem}\label{thm:Kelly-dim-bounded-main}
There is a function\/ $f\colon\setN\times\setN\to\setN$ such that for all integers\/ $t\geq 1$ and\/ $k\geq 3$, every poset with cover graph of treewidth less than\/ $t$ and dimension greater than\/ $f(t,k)$ contains\/ $\kelly_k$ as a subposet.
\end{theorem}

At first sight, it might not be clear why forcing Kelly examples, which is more difficult than forcing standard examples, is worth the extra effort.
However, and rather unexpectedly, \Cref{thm:Kelly-dim-bounded-main} turns out to be the key tool to answer the following question:

\begin{question}
What are the minor-closed graph classes\/ $\cgC$ such that posets with cover graphs in\/ $\cgC$ have bounded dimension?
\end{question}

This question has been thoroughly studied over the last years.
Already in 1977, Trotter and Moore~\cite{TM77} showed that dimension is at most $3$ when $\cgC$ is the class of forests.
Felsner, Trotter, and Wiechert~\cite{FTW13} proved that dimension is at most $4$ when $\cgC$ is the class of outerplanar graphs.
When $\cgC$ is the class of graphs with pathwidth $2$, Biró, Keller, and Young~\cite{BKY16} showed that dimension is at most $17$, which was later improved to $6$ by Wiechert~\cite{Wiechert17}.
When $\cgC$ is the class of graphs with treewidth $2$, Joret, Micek, Trotter, Wang, and Wiechert~\cite{JMTWW17} proved that dimension is at most $1276$, which was later improved to $12$ by Seweryn~\cite{Seweryn20}.
Seweryn~\cite{Seweryn23} showed that dimension is bounded by a function of $t$ when $\cgC$ is the class of $K_{2,t}$-minor-free graphs.
Gorsky and Seweryn~\cite{GS21} showed that dimension is bounded by a function of $k$ when $\cgC$ is the class of $k$-outerplanar graphs.
Most recently, Huynh, Joret, Micek, Seweryn, and Wollan~\cite{HJMSW22} showed that dimension is bounded by a function of $t$ when $\cgC$ is the class of graphs excluding a ladder with $t$ rungs (i.e., a $2\times t$ grid) as a minor.
Yet, despite all these partial results, solving the problem completely seemed out of reach until now.%
\footnote{For instance, even showing that excluding a $2$-sum of two $K_4$s suffices to bound the dimension was open.}
Using \Cref{thm:Kelly-dim-bounded-main}, we fully answer the question as follows.

\begin{theorem}\label{thm:excluding-a-minor}
Let\/ $\cgC$ be a minor-closed class of graphs.
The posets with cover graphs in\/ $\cgC$ have dimension bounded from above by a constant depending on\/ $\cgC$ if and only if\/ $\cgC$ excludes the cover graph of\/ $\kelly_k$ for some\/ $k\geq 3$.
\end{theorem}

Thus, the posets constructed by Kelly in 1981 turned out to be a central piece of the puzzle.
It is also worth recasting \Cref{thm:excluding-a-minor} in the context of the question of whether posets with cover graphs in a fixed proper minor-closed class $\cgC$ are $\dim$-bounded.
While this is very much open, \Cref{thm:excluding-a-minor} makes a first step towards a solution, by showing that if the posets with cover graphs in $\cgC$ have bounded standard example number, then the posets do have bounded dimension.
(Indeed, if $\cgC$ contains the cover graphs of all Kelly examples, then both the standard example number and dimension are unbounded, while if $\cgC$ excludes such a graph, then both are bounded by \Cref{thm:excluding-a-minor}.)

The last contribution of this paper is about a variant of poset dimension called Boolean dimension.
Realizers provide a compact scheme for handling comparability queries: given a realizer $\set{L_1,\dots,L_k}$ for a poset $P$, a query of the form ``is $x\leq y$?'' can be answered by looking at the relative position of $x$ and $y$ in each of the $k$ linear extensions of the realizer.
This motivates the following more general encoding framework for posets, introduced by Gambosi, Nešetřil, and Talamo~\cite{GNT90}.
A \emph{Boolean realizer} of a poset $P$ is a sequence $(L_1,\ldots,L_k)$ of linear orders on the elements of $P$ (which are not necessarily linear extensions of $P$) and a $k$-ary Boolean function $\phi$ such that
\[x\leq y\text{ in }P\quad\text{if and only if}\quad\phi\bigl((x\leq y\text{ in }L_1),\:\ldots,\:(x\leq y\text{ in }L_k)\bigr)=1{,}\]
for all distinct $x,y\in P$.
The \emph{Boolean dimension} of $P$, denoted by $\bdim(P)$, is the minimum size of a Boolean realizer of $P$.
Clearly, for every poset $P$, we have $\bdim(P)\leq\dim(P)$.

In 1989, Nešetřil and Pudlák~\cite{NP89} posed a beautiful open problem:%
\footnote{The original question was in fact posed for posets with planar diagrams, which are more restricted.}

\begin{question}
Do posets with planar cover graphs have bounded Boolean dimension?
\end{question}

It could even be the case that for every proper minor-closed graph class $\cgC$, the posets with cover graphs in $\cgC$ have bounded Boolean dimension.
So far, it is known that posets with planar cover graphs and a unique minimal element have Boolean dimension at most $13$, as proved by Blake, Micek, and Trotter~\cite{BMT22}.
Also, Felsner, Mészáros, and Micek~\cite{FMM20} proved that posets with cover graphs of bounded treewidth have bounded Boolean dimension.
Using our framework, we prove a more general result with an arguably much simpler proof.

\begin{theorem}\label{thm:bdim-main}
Posets of bounded cliquewidth have bounded Boolean dimension.
\end{theorem}

In the special case of posets with cover graphs of treewidth $t$, the result from~\cite{FMM20} gives a bound of $\cramped{2^{2^{\Oh(t)}}}$ on the Boolean dimension, while our proof gives a bound of $2^{\cramped{\Oh(t^2)}}$; see \Cref{thm:bdim-treewidth}.

We remark that nothing is specific to posets in the definition of Boolean dimension.
It can be applied to arbitrary directed graphs, where $x\leq y$ (for distinct $x$ and $y$) would mean an edge from $x$ to $y$, and loops would be ignored.
\Cref{thm:bdim-main} works just as well in that more general setting (with the same proof), where it asserts that directed graphs of bounded cliquewidth have bounded Boolean dimension.
In the rest of this paper, following the tradition established in the literature, we define the respective terms and prove \Cref{thm:bdim-main} just for posets.

The paper is organized as follows.
First, we provide an informal overview of our proofs in \Cref{sec:overview}, emphasizing the main ideas and in particular introducing Colcombet's Theorem.
Then, in \Cref{sec:prelims}, we give the necessary definitions and preliminaries.
We prove \Cref{thm:dim-bounded-main}, \Cref{thm:Kelly-dim-bounded-main}, \Cref{thm:excluding-a-minor}, and \Cref{thm:bdim-main} in \Cref{sec:dim-bnd}, \Cref{sec:kelly}, \Cref{sec:minors}, and \Cref{sec:bdim}, respectively.

\section{Overview}\label{sec:overview}

In this section, we provide a more in-depth overview of the contributions of this work, aimed at giving a high-level perspective on our arguments and highlighting the main new ideas.
We note that this section is focused only on providing intuition and can be freely skipped---it is not necessary to follow the rest of the paper.

\subsection*{NLC-width}

We start with recalling the definition of NLC-width, which is a parameter introduced by Wanke~\cite{Wanke94} that is functionally equivalent to cliquewidth.
Following Gurski, Wanke, and Yilmaz~\cite{GWY16}, an \emph{NLC-decomposition} of a directed graph $(X,E)$ is a tuple $(T,Q,\eta,\rho,R,R')$ where $T$ is a binary tree with leaf set $X$ and $Q$ is a set of \emph{labels} such that
\begin{itemize}
\item every leaf $x\in X$ is assigned some \emph{initial label} $\initial{x}\in Q$;
\item every edge $uv$ of $T$, where $u$ is the parent of $v$, is assigned a \emph{relabeling} $\bagbb{u}{v}$ which is a function from $Q$ to $Q$; and
\item every node $u$ of $T$ is assigned two \emph{status relations} $R(u),R'(u)\subseteq Q\times Q$.
\end{itemize}
We extend the relabeling assignment by defining $\bagbb{u}{v}$ for every pair of nodes $u$ and $v$ such that $u$ is an ancestor of $v$ as the composition $\bagbb{u_0}{u_1}\cdots\bagbb{u_{\ell-1}}{u_\ell}$ where $u_0u_1\cdots u_\ell$ is the path from $u$ to $v$ in $T$.
Intuitively, $\bagbb{u}{v}$ describes how the labels at $v$ are transformed into the labels at $u$.
For a leaf $x\in X$ and an ancestor $u$ of $x$, we define the label of $x$ at $u$, denoted by $\bagbl{u}{x}$, as the image of the initial label $\initial{x}$ under the relabeling $\bagbb{u}{x}$.
So the relabeling $\bagbb{u}{v}$ can be thought of as transforming the label of $x$ at $v$ into the label of $x$ at $u$.
Finally, we require that the decomposition encodes $(X,E)$ in the following sense: for all distinct $x,y\in X$, if $u$ is the lowest common ancestor of $x$ and $y$ in $T$, $x$ is a descendant of the left child of $u$, and $y$ is a descendant of the right child of $u$, then
\begin{itemize}
\item $(x,y)\in E$ if and only if $(\bagbl{u}{x},\bagbl{u}{y})\in R(u)$, and
\item $(y,x)\in E$ if and only if $(\bagbl{u}{x},\bagbl{u}{y})\in R'(u)$.
\end{itemize}
The \emph{NLC-width} of a directed graph $(X,E)$ is the least $q$ such that $(X,E)$ admits an NLC-decomposition that uses a label set of size $q$.
This definition can be then applied to posets by treating them as directed graphs in the natural manner.

A reader familiar with the notion of cliquewidth will immediately recognize similar mechanics.
Namely, one may imagine that an NLC-decomposition as above constructs $(X,E)$ in a bottom-up manner, starting with single vertices.
At every point, every vertex of the currently constructed part of $(X,E)$ carries some label from a fixed label set $Q$.
Every step of the construction consists in either applying some relabeling function to the vertices or joining two already constructed parts of $(X,E)$ while creating arcs between them.
The status relations $R$ and $R'$ determine what arcs should be put between vertices from the joined parts, depending on their current labels.
As proved by Gurski, Wanke, and Yilmaz~\cite{GWY16}, the notion of NLC-width of directed graphs defined above differs from the cliquewidth by at most a multiplicative factor of $2$.

\subsection*{Colcombet's Theorem}

We now recall the theorem of Colcombet, which is the fundamental tool underlying our main results.
First, we need a few definitions.

Consider a rooted tree $T$ whose edges are labeled with elements of a finite semigroup $(\Lambda,\cdot)$.
We let $u\downtoeq v$ and $u\downto v$ denote that $u$ is an ancestor of $v$ in $T$ allowing $u=v$ in the first but not in the second case.
For nodes $u\downto v$, we let $\lambda(u,v)$ denote the product of the labels on edges along the $u$-to-$v$ path in $T$.

The decomposition of $T$ provided by Colcombet's Theorem has the form of a \emph{split} of $T$, which is a function $s$ from the inner (non-root, non-leaf) nodes of $T$ to $\set{1,\ldots,n}$, where $n$ is the \emph{order} of the split.
The reader should think of a split as of a hierarchical factorization of $T$: first split $T$ into subtrees along inner nodes with value $n$, then split those subtrees into even smaller subtrees along inner nodes with value $n-1$, and so on.
All subtrees of the factorization can be arranged into \emph{levels} $0,\ldots,n$ so that level $0$ consists of the trivial (unsplittable) subtrees, a subtree $S$ at any level $h\in\set{1,\ldots,n}$ is split by its inner nodes of value $h$ into subtrees at level $h-1$ called the \emph{factors} of $S$ at that level, and level $n$ contains only the whole tree $T$.
(Note that $S$ is a factor of itself at level $h$ if it has no inner nodes of value $h$.)
For $h\in\set{1,\ldots,n}$, we call two nodes $u\downto v$ \emph{$h$-neighbors} if $s(u)=s(v)=h$ and every node $w$ on the $u$-to-$v$ path in $T$ satisfies $s(w)\leq h$.
Intuitively, this means that splitting at $u$ and $v$ happens at the same step of the factorization.

The key Ramseyan property exposed by Colcombet's Theorem is as follows.
A split $s$ is \emph{forward Ramseyan} if for all $h\in\set{1,\ldots,n}$ and all pairs of $h$-neighbors $u\downto v$ and $u'\downto v'$, we have
\[\lambda(u,v)\cdot\lambda(u',v')=\lambda(u,v){.}\]
It is instructive to consider how this condition works when $u\downto v\downto w$ are three $h$-neighbors in a row.
Then the upper element $\lambda(u,v)$ absorbs the lower element $\lambda(v,w)$:
\[\lambda(u,w)=\lambda(u,v)\cdot\lambda(v,w)=\lambda(u,v){.}\]

With these definitions in place, we can now state Colcombet's Theorem, originally proved in~\cite{Colcombet07}.
For the reader's convenience, we provide a self-contained proof of the theorem in \Cref{sec:colcombet}.

\begin{colcombet}
For every rooted tree\/ $T$ edge-labeled with a finite semigroup\/ $(\Lambda,\cdot)$, there exists a split of\/ $T$ of order at most\/ $\size{\Lambda}$ that is forward Ramseyan.
\end{colcombet}

Again, the reader may think of the split provided by Colcombet's Theorem as of a hierarchical factorization, where the values assigned by the split correspond to the levels.
The condition that the split is forward Ramseyan provides the aforementioned absorption property among the newly created factors at every step of the factorization.
Importantly, the number of levels of the factorization is bounded by $\size{\Lambda}+1$, independently of the particular tree $T$.
This allows applying inductive schemes on the factorization with the total number of steps bounded in terms of $\size{\Lambda}$.

Now, consider an arbitrary NLC-decomposition $(T,Q,\eta,\rho,R,R')$.
It is natural to view the rooted tree $T$ as edge-labeled with the relabelings $\rho$, treated as elements of the semigroup of functions from $Q$ to $Q$ with composition as the semigroup operation.
Thus, Colcombet's Theorem provides a split of order bounded by $\size{Q}^{\size{Q}}$.
The forward Ramseyan property asserts the following: if $E_h$ is the set of elements of the form $\bagbb{u}{v}$ for all pairs of $h$-neighbors $u\downto v$, then $\rho\sigma=\rho$ for all $\rho,\sigma\in E_h$.
It is not difficult to see that in terms of functions, this condition can be rephrased as follows: there is a fixed partition $\cgP$ of $Q$ such that every function in $E_h$ acts on $Q$ by mapping every part $A\in\cgP$ to a single element of $A$.
This imposes a very strong structure on the split, which can be used in various inductive arguments.

In their work on polynomial $\chi$-boundedness of graphs of bounded cliquewidth, Bonamy and Pilipczuk~\cite{BP20} used exactly this edge-labeling of an NLC-decomposition of a graph as the base for an application of Colcombet's Theorem.
This is also what we do in the context of Boolean dimension in the proof of \Cref{thm:bdim-main}.
However, for the proof of \Cref{thm:dim-bounded-main}, we need to enrich the considered semigroup with more information.
Intuitively, for $u\downto v$, the semigroup element $\lambda(u,v)$ associated with the $u$-to-$v$ path in $T$ encodes both the relabeling $\bagbb{u}{v}$ and the existence of certain building blocks for standard examples among descendants of $u$ that are not descendants of $v$.
We explain this in more details next.

\subsection*{Sketch of the proof of Theorem \ref{thm:dim-bounded-main}}

We want to prove that for all positive integers $q$ and $k$, every poset with NLC-width $q$ and sufficiently large dimension contains the standard example $S_k$ as a subposet.
Given such a poset $(X,\leq)$ and its NLC-decomposition with label set of size $q$, the proof proceeds in two steps.

\medskip
\textit{Step 1: Coloring incomparable pairs}.
First, we introduce some auxiliary notation.
Fix an element $x\in X$.
The comparability status of $x$ to an (unknown) element $y\in X$ is determined by the label $\bagbl{u}{y}$ of $y$ at the lowest common ancestor $u$ and $x$ and $y$ in $T$, which is then determined by the label $\bagbl{v}{y}$ of $y$ at any node $v$ between $y$ and $u$ in $T$.
This allows us to think of comparability status of $x$ to a label (rather than a specific element) at any node $v$ that is not an ancestor of $x$ in $T$.
Let $\lablr{x}{v}$ be the set of labels $\gamma$ such that ``$x\leq\gamma$ at $v$'', and let $\labgr{x}{v}$ the set of labels $\gamma$ such that ``$x\geq\gamma$ at $v$''.
Thus, whenever $v\downtoeq y\in X$, we have $x\leq y$ if and only if $\bagbl{v}{y}\in\lablr{x}{v}$, and similarly $x\geq y$ if and only if $\bagbl{v}{y}\in\labgr{x}{v}$.

A standard approach to bounding the dimension of a poset relies on the following hypergraph coloring view of dimension.
An \emph{alternating cycle} is a sequence of incomparable pairs $(x_1,y_1),\ldots,(x_m,y_m)$, for some $m\geq 2$, such that $x_i\leq y_{i+1}$ for all $i\in[m]$ (cyclically, so $x_m\leq y_1$) and these are the only comparabilities among $x_1,\ldots,x_m,y_1,\ldots,y_m$.\footnote{In the literature, such a sequence is usually called a \emph{strict alternating cycle}, while an \emph{alternating cycle} is usually defined without the latter condition.
However, we will not consider non-strict alternating cycles.}
It is well known (and easy to see) that the dimension of the poset is equal to the chromatic number of the hypergraph defined on the set of incomparable pairs, having one hyperedge for each alternating cycle.
In other words, in order to bound the dimension from above by $d$, it is enough to color the incomparable pairs with at most $d$ colors so that no alternating cycle becomes monochromatic.

In our proof, we set out to do precisely that, coloring the incomparable pairs using a number of colors bounded by a function of the NLC-width $q$ and of $k$.
We cannot always succeed in avoiding monochromatic alternating cycles (e.g., when the dimension is greater than the number of colors that we use).
The coloring is designed so that if we fail, then we exhibit two structures, a ``long chain'' and a ``cross'', that allow us, in Step~2, to find a standard example $S_k$ as a subposet.

Assume we are given a split of some order $p$, bounded by a function of $q$ and $k$, that is forward Ramseyan with respect to the labeling of $T$ by the relabeling assignment $\rho$.
The existence of such a split is guaranteed by Colcombet's Theorem.
(The actual labeling of $T$ on which we apply Colcombet's Theorem contains more information, but we skip this detail for now and bring it back in Step~2.)
As discussed before, the split provides a factorization of $T$ into subtrees at levels $0,\ldots,p$.
We say that an element or a set of elements of $X$ is \emph{enclosed} in such a subtree if they are all descendants of the root of that subtree.

We consider the subtrees appearing at each level of the factorization, and for each subtree $S$, we consider the alternating cycles whole elements are all descendants of the root of $S$ but are not all descendants of the same leaf of $S$.
At each level, and for each alternating cycle $C$, these conditions are satisfied for a unique subtree $S$, namely, the lowest subtree enclosing $C$ at that level.
For brevity, we say that $C$ is an alternating cycle of $S$.

For every subtree $S$ in the factorization, we aim at a coloring of the incomparable pairs enclosed in $S$ that makes no alternating cycle of $S$ monochromatic.
We try to construct these colorings by induction on the factorization, starting with the trivial trees at the bottom level and ending with the entire $T$.
We succeed with the construction unless we find a long chain and a cross at some level.
If we succeed, then the final coloring for $T$ colors all incomparable pairs without making any alternating cycle monochromatic, so the dimension is bounded by the number of colors used.
For fixed constants $q$ and $k$, the number of colors is bounded at level $0$ and grows with every level as a function of the number of colors used at the previous level.
Therefore, since the number of levels is bounded in terms of $q$ and $k$, so is the total number of colors used by the coloring for $T$.

At level $0$ of the factorization, each subtree $S$ is trivial---it has a root $u$ and two leaves $v$ and $w$ that are the left and the right children of $u$.
All incomparable pairs of descendants of $v$ can be assigned one color, because alternating cycles made of such pairs are not alternating cycles of $S$.
Likewise, all incomparable pairs of descendants of $w$ can be assigned one color (different from the previous one).
For the remaining incomparable pairs, the way how the order relation $\leq$ is defined from the NLC-decomposition implies that the comparability status of $x$ and $y$ is determined by $\bagbl{u}{x}$ and $\bagbl{u}{y}$.
In particular, the incomparable pairs $(x,y)$ with $x$ a descendant of $v$ and $y$ a descendant of $w$ can be colored according to the value of $\bagbl{u}{x}$, because in any alternating cycle of $S$ formed by such pairs, the values of $\bagbl{u}{x}$ are pairwise distinct.
Thus, $q$ colors suffice for the pairs with $x$ a descendant of $v$ and $y$ a descendant of $w$, and vice versa.
Altogether, we have a coloring with $2q+2$ colors that makes no alternating cycle of $S$ monochromatic.

Now, consider a subtree $S$ at level $h\geq 1$.
Let $\cgF$ be the family of factors of $S$ at level $h-1$.
Assume that for every $S'\in\cgF$, we have a coloring of the incomparable pairs enclosed in $S'$ (constructed at level $h-1$) that makes no alternating cycles of $S'$ monochromatic.
We provide two ways of constructing a coloring for $S$, one suitable when the ``height'' of $\cgF$ is small and another suitable when it is large.

In the first coloring, the color of an incomparable pair $(x,y)$ comprises the colors of $(x,y)$ for all subtrees in $\cgF$ that enclose $x$ and $y$ (ordered from the top to the bottom such subtree).
Consider an alternating cycle $C$ of $S$, and let $\bar S$ be the lowest subtree in $\cgF$ enclosing all elements of $C$.
Thus, $C$ is an alternating cycle of $\bar S$, so it is not monochromatic in the coloring for $\bar S$.
Since for every incomparable pair, its color for $S$ includes its color for $\bar S$, the cycle $C$ is not monochromatic in the coloring for $S$ either.

The first coloring is very strong, but in order to have a bounded number of colors, we can use it only when every incomparable pair is enclosed in a bounded number of subtrees in $\cgF$.
If some incomparable pair is enclosed in at least $k+2$ subtrees in $\cgF$, then the roots of any $k+1$ of them other than the top one form a sequence $u_0\downto\cdots\downto u_k$ of mutual $h$-neighbors.
We call such a sequence an \emph{$h$-chain} of length $k$.
It is one of the two structures used for constructing a standard example in Step~2.

An alternative way of coloring the incomparable pairs enclosed in $S$ uses a number of colors that is independent of the ``height'' of $\cgF$.
It makes essential use of the forward Ramseyan property of the set $E_h$ of relabelings between $h$-neighbors.
For the purpose of this overview, we present a suitable coloring under the simplifying assumption that $E_h=\set{\id_Q}$ (where $\id_Q$ is the identity on the label set $Q$) and our only goal is to make alternating cycles of length $2$ monochromatic (disregarding longer alternating cycles).
The key idea shows up already in this case.
The color of an incomparable pair $(x,y)$ comprises the following information:
\begin{itemize}
\item the color of $(x,y)$ for the top subtree in $\cgF$,
\item the color of $(x,y)$ for the lowest subtree in $\cgF$ enclosing $x$ and $y$,
\item the labels $\bagbl{v}{x}$ and $\bagbl{v}{y}$ of $x$ and $y$ at the root $v$ of the lowest subtree in $\cgF$ enclosing $x$ and $y$.
\end{itemize}
Since $E_h=\set{\id_Q}$, if $u$ is the root of any subtree in $\cgF$ enclosing $x$ and $y$ other than the top one, then the labels $\bagbl{u}{x}$ and $\bagbl{u}{y}$ are equal to the labels included in the color of $(x,y)$, as
\[\bagbl{u}{x}=\bagbb{u}{v}\bagbl{v}{x}=\id_Q\bagbl{v}{x}=\bagbl{v}{x} \qquad\text{and}\qquad \bagbl{u}{y}=\bagbl{v}{y}{.}\]

Now, consider what happens if this coloring assigns the same color to two incomparable pairs $(x_1,y_1)$ and $(x_2,y_2)$ that form an alternating cycle of $S$.
Let $\bar S$ be the lowest subtree in $\cgF$ that encloses both pairs.
Thus, not all $x_1$, $y_1$, $x_2$, and $y_2$ are descendants of the same leaf of $\bar S$.
The subtree $\bar S$ cannot be the top subtree in $\cgF$, because the colors for $S$ include the colors for the top subtree in $\cgF$.
Let $u$ be the root of $\bar S$.
The labels $\bagbl{u}{x_i}$ and $\bagbl{u}{y_i}$ are included in the coloring, so they are common for both $i\in\set{1,2}$, equal to some $\alpha$ and $\beta$ respectively.

Since the coloring for $S$ includes the color of $(x_i,y_i)$ for the lowest subtree in $\cgF$ enclosing $x_i$ and $y_i$, that subtree cannot be $\bar S$ for both $i\in\set{1,2}$.
So assume, without loss of generality, that $\bar S$ is not the lowest subtree in $\cgF$ enclosing $x_1$ and $y_1$.
Thus, $x_1$ and $y_1$ are descendants of the same leaf $v$ of $\bar S$ which is not a leaf of $S$.
It follows that $\bagbl{v}{x_1}=\alpha$ and $\bagbl{v}{y_1}=\beta$.
If $x_2$ is also a descendant of $v$ while $y_2$ is not, then $\bagbl{v}{x_2}=\alpha$ as well, so \emph{$y_2$ has the same comparability status to both\/ $x_1$ and\/ $x_2$}, which contradicts the assumption that $(x_1,y_1)$ and $(x_2,y_2)$ form an alternating cycle.
An analogous contradiction is reached when $y_2$ is a descendant of $v$ while $x_2$ is not.
Therefore, neither $x_2$ nor $y_2$ is a descendant of $v$.

A critical use of the fact that the order relation $\leq$ is defined from the NLC-decomposition is made above in the emphasized statement that $y_2$ has the same comparability status to both $x_1$ and $x_2$.
Using the notation introduced before, this comparability status is determined by whether or not $\alpha\in\labgr{y_2}{v}$.
Since neither $x_2$ nor $y_2$ is a descendant of $v$, we have $\alpha\in\labgr{y_2}{v}$ (as $y_2\geq x_1$) and $\beta\in\lablr{x_2}{v}$ (as $x_2\leq y_1$).
We have thus achieved the following key structure used in Step~2 (along with a long $h$-chain) for constructing a standard example: two $h$-neighbors $u\downto v$ and an incomparable pair of descendants $x$ and $y$ of $u$ but not $v$ such that
\[\bagbl{u}{x}\in\labgr{y}{v} \qquad\text{and}\qquad \bagbl{u}{y}\in\lablr{x}{v}{.}\]
We call it an \emph{$h$-cross}, and let us call the quadruple
\[\bigl(\bagbl{u}{x},\:\bagbl{u}{y},\:\lablr{x}{v},\:\labgr{y}{v}\bigr)\]
the \emph{signature} of the $h$-cross.
(The actual definition of $h$-cross used in the proof is slightly more technical, as it needs to accommodate the case of arbitrary $E_h$.)

Altogether, if there is no $h$-chain of length $k$, then we avoid any monochromatic alternating cycles of $S$ by applying the first coloring, which then combines at most $k+1$ colors from level $h-1$, thus using a bounded number of colors in total.
If there is an $h$-chain of length $k$, then we apply the second coloring (which combines two colors from level $h-1$ and two labels) to avoid any monochromatic alternating cycles of $S$ unless there is an $h$-cross.
If we have both an $h$-chain of length $k$ and an $h$-cross, then we proceed to Step~2 in order to find a standard example $S_k$ as a subposet.

\medskip
\textit{Step 2: Building the standard example}.
In Step~1, we have found an $h$-chain $u_0\downto\cdots\downto u_k$ and an $h$-cross with some signature $(\alpha,\beta,B,A)$ such that $\alpha\in A$ and $\beta\in B$.
Suppose that we can find an $h$-cross with the same signature between every consecutive pair of $h$-neighbors in the $h$-chain $u_0\downto\cdots\downto u_k$.
That is, suppose for every $j=1,\ldots,k$, we have an incomparable pair of descendants $x_j$ and $y_j$ of $u_{j-1}$ but not $u_j$ such that
\[\bagbl{u_{j-1}}{x_j}=\alpha\in A=\labgr{y_j}{u_j} \qquad\text{and}\qquad \bagbl{u_{j-1}}{y_j}=\beta\in B=\lablr{x_j}{u_j}{.}\]
Suppose further that $E_h=\set{\id_Q}$, as we did when discussing Step~1.
Then $x_1,\ldots,x_k$ and $y_1,\ldots,y_k$ form a standard example, because whenever $1\leq i<j\leq k$, we have
\begin{alignat*}{7}
\bagbl{u_i}{x_j}&=\bagbb{u_i}{u_{j-1}}\bagbl{u_{j-1}}{x_j}&&=\id_Q\alpha&&=\alpha&&\in A&&=\labgr{y_i}{u_i} &&\quad\text{and so}\quad &y_i&\geq x_j{,} \\
\bagbl{u_i}{y_j}&=\bagbb{u_i}{u_{j-1}}\bagbl{u_{j-1}}{y_j}&&=\id_Q\beta&&=\beta&&\in B&&=\lablr{x_i}{u_i} &&\quad\text{and so}\quad &x_i&\leq y_j{.}
\end{alignat*}

To be able to argue about existence of specific $h$-crosses between arbitrary pairs of $h$-neighbors, before applying Colcombet's Theorem, we enrich the semigroup and the labeling of $T$ with additional information.
For any two nodes $u\downto v$, let $\lambda(u,v)=(\bagbb{u}{v},\Phi(u,v),\Psi(u,v))$ where
\begin{alignat*}{2}
\Phi(u,v) &= \set[\big]{\bigl(\bagbl{u}{x},\:\lablr{x}{v},\:\labgr{x}{v}\bigr)\colon x\text{ is a descendant of }u\text{ but not }v}{,}\hskip-200pt \\
\Psi(u,v) &= \smash[b]{\bigl\{\bigl(\bagbl{u}{x},\:\bagbl{u}{y},\:\lablr{x}{v},\:\labgr{y}{v}\bigr)}\colon &&x\text{ and }y\text{ are incomparable and} \\[-\jot]
&&&\text{are descendants of }u\text{ but not }v\smash[t]{\bigr\}}{,}
\end{alignat*}
and let $\Lambda$ be the set of all triples $(\rho,\Phi,\Psi)$ of the appropriate type.
It is then possible to provide appropriate formulas for the semigroup operation $\cdot$ on $\Lambda$ so that if $u\downto v\downto w$, then $\lambda(u,w)=\lambda(u,v)\cdot\lambda(v,w)$.
The purpose of $\Psi$ is to encode the existence of $h$-crosses with specific signatures.
The purpose of $\Phi$ is that an incomparable pair witnessing some quadruple in $\Psi(u,w)$ may have one element ``between $u$ and $v$'' and the other ``between $v$ and $w$'', and in that case the information about these two elements is encoded in $\Phi(u,v)$ and $\Phi(v,w)$, respectively.
Note that the size of $\Lambda$ is bounded by a function of $q$.

We apply Colcombet's Theorem to obtain a split $s$ of $T$ which is forward Ramseyan with respect to $\lambda$.
We assume that Step~1 has been performed with this particular split, resulting in an $h$-chain $u_0\downto\cdots\downto u_k$ and an $h$-cross with some signature $(\alpha,\beta,B,A)$ such that $\alpha\in A$ and $\beta\in B$ between some pair of $h$-neighbors $u$ and $v$.
It follows that $(\alpha,\beta,B,A)\in\Psi(u,v)$.
To complete this overview under the simplifying assumption that $E_h=\set{\id_Q}$, observe that for every $j\in\set{1,\ldots,k}$, the property that $s$ is forward Ramseyan with respect to $\lambda$ yields $\lambda(u_{j-1},u_j)\cdot\lambda(u,v)=\lambda(u_{j-1},u_j)$, which together with $\bagbb{u_{j-1}}{u_j}=\id_Q$ implies $\Psi(u,v)\subseteq\Psi(u_{j-1},u_j)$, because this is how the operation $\cdot$ is (and must be) defined.
We conclude that $(\alpha,\beta,B,A)\in\Psi(u_{j-1},u_j)$ for every $j$, which gives us $h$-crosses with signature $(\alpha,\beta,B,A)$ between all consecutive pairs of $h$-neighbors in the $h$-chain $u_0\downto\cdots\downto u_k$.
These $h$-crosses, as we have already discussed, give rise to the standard example $S_k$ as a subposet.

\subsection*{Kelly examples: Theorems \ref{thm:Kelly-dim-bounded-main} and \ref{thm:excluding-a-minor}}

The proof of \Cref{thm:Kelly-dim-bounded-main} relies on the same two-step procedure as the proof of \Cref{thm:Kelly-dim-bounded-main} but applied to an NLC-decomposition of the poset obtained from a tree decomposition of its cover graph via a particular construction.
Thus, $T$ is a binary tree with leaf set $X$ (the ground set of the poset), every node $u$ of $T$ is equipped with a \emph{bag} $B(u)\subseteq X$ of size at most $t$ in a way that forms a tree decomposition of the cover graph, and moreover $x\in B(x)$ for every $x\in X$.
Such a tree decomposition of optimal width always exists.
A key \emph{convexity property} is that whenever $B(u)\ni x\leq y\in B(v)$ and $w$ is a node on the $u$-to-$v$ path in $T$, then there is $z\in B(w)$ such that $x\leq z\leq y$.

We can enumerate the elements of every bag as $B(u)=\set{z^u_1,\ldots,z^u_t}$ (possibly repeating some elements if the bag has size less than $t$).
Then, using the convexity property, we can derive an NLC-decomposition of the poset so that the label $\bagbl{u}{x}$ of $x\in X$ at an ancestor $u$ describes how $x$ ``sees'' the bag of $u$ in the order:
\[\bagbl{u}{x}=\bigl(\set{i\in\set{1,\ldots,t}\colon z^u_i\geq x},\;\set{i\in\set{1,\ldots,t}\colon z^u_i\leq x}\bigr){.}\]
Thus, the label set $Q$ consists of the pairs of subsets of $\set{1,\ldots,t}$; in particular $\size{Q}=4^t$.

When using an $h$-chain $u_0\downto\cdots\downto u_k$ and an $h$-cross to build a standard example on incomparable pairs $x_1,y_1,\ldots,x_k,y_k$ as described in Step~2, we exploit this specific form of the NLC-decomposition to find chains $c_1\leq\cdots\leq c_{k-1}$ and $d_1\geq\cdots\geq d_{k-1}$ with $c_j,d_j\in B(u_j)$ such that whenever $1\leq i<j\leq k$, the order relations between $x_i,y_i$ and $x_j,y_j$ are witnessed as
\[x_i\leq c_i\leq\cdots\leq c_{j-1}\leq y_j\qquad\text{and}\qquad y_i\geq d_i\geq\cdots\geq d_{j-1}\geq x_j{.}\]
This is enough to exhibit $\kelly_k$ as a subposet.

As for \Cref{thm:excluding-a-minor}, the difficult implication is to prove that if a class of graphs $\cgC$ excludes the cover graph of some Kelly example as a minor, then the posets with cover graphs from $\cgC$ have bounded dimension.
Since the cover graph of every Kelly example is planar, the Grid Minor Theorem implies that any class $\cgC$ as above has bounded treewidth.
So from \Cref{thm:Kelly-dim-bounded-main} we infer that posets with cover graphs in $\cgC$ exclude some Kelly example as a subposet.

To complete the argument, we need the following statement: for every $k$, there exists $N$ such that if a poset $P$ contains a $\kelly_N$ as a subposet, then the cover graph of $P$ contains the cover graph of $\kelly_k$ as a minor.
In the proof, we start with a copy of $\kelly_N$ in $P$ for some large $N$, and we investigate paths in the cover graph of $P$ that witness the relations in the copy.
These paths may in principle intersect, but certain intersections can be excluded, because they would imply non-existent relations in $\kelly_N$.
We analyze the structure of the intersections using Ramsey's Theorem to either ``uncross'' the paths and immediately construct a minor model of the cover graph of $\kelly_k$, or expose a bramble of large order in the cover graph of $P$.
The latter conclusion implies that the treewidth of the cover graph is large, which again yields a minor model of $\kelly_k$ by the Grid Minor Theorem.

\subsection*{Boolean dimension: Theorem \ref{thm:bdim-main}}

Finally, we sketch the proof of \Cref{thm:bdim-main}.
By considering an NLC-decomposition as a rooted tree edge-labeled with relabelings $\rho$, the problem can be reduced to proving the following statement.
Suppose we have a rooted tree $T$ whose edges are labeled with elements of some finite semigroup $(\Lambda,\cdot)$.
For $u\downto v$, let $\lambda(u,v)$ denote the product of the edge labels along the $u$-to-$v$ path in $T$.
Let $X$ be the leaf set of $T$.
We want to construct linear orders $\preceq_1,\ldots,\preceq_p$ on $X$, for some $p$ depending only on $\size{\Lambda}$, so that for given $x,y\in X$, from the $p$-tuple of Boolean values $[x\preceq_1y],\ldots,[x\preceq_py]$ one can uniquely determine the values $\lambda(u,x)$ and $\lambda(u,y)$, where $u$ is the lowest common ancestor of $x$ and $y$ in $T$.

To prove this statement, we use Colcombet's Theorem.
To start the discussion, consider the special case when the entire set $\Lambda$ is forward Ramseyan: $e\cdot f=e$ for all $e,f\in\Lambda$.
Then, in the context of a query about $x,y\in X$ as above, we have $\lambda(u,x)=\lambda(u,v_x)$, where $v_x$ is the child of $u$ that is an ancestor of $x$.
Therefore, the problem simplifies significantly: instead of determining the product of all semigroup elements along the $u$-to-$x$ path, we need to determine only the single semigroup element placed at the edge connecting $u$ with the child $v_x$ ``in the direction of $x$''.
For this basic \emph{color detection} problem, one can give a direct construction requiring as few as $\Oh(\log\size{\Lambda})$ linear orders $\preceq_i$.
We remark that such a construction was already observed by Felsner, Mészáros, and Micek \cite[Color Detection]{FMM20}.

Once the special case is understood, we can proceed with the proof of the general case using induction on the split $s$ of order $\size{\Lambda}$ provided by Colcombet's Theorem.
In step $h\in\set{0,\ldots,\size{\Lambda}}$ of the induction, we construct a set of $\Oh((h+1)\log\size{\Lambda})$ linear orders which determine, for any two distinct nodes $x,y\in X$, the value of $\lambda(u,v)$, where $u$ is the lowest common ancestor of $x$ and $y$ in $T$ and $v$ is $x$ or the first node on the $u$-to-$x$ path such that $s(v)>h$.
In particular, when $h=0$, we use $\Oh(\log\size{\Lambda})$ linear orders to determine $\lambda(u,v_x)$, as described for the special case.
When $h\geq 1$, the question about the value of $\lambda(u,v)$ is divided into three questions: about the values of $\lambda(u,u')$, $\lambda(u',u'')$, and $\lambda(u'',v)$, where $u'$ and $u''$ are respectively the first and the last node with value $h$ on the $u$-to-$v$ path in $T$.
The question about $\lambda(u,u')$ is handled in the previous step of the induction, answering the question about $\lambda(u',u'')$ boils down to applying (a more involved variant of) the special case thanks to forward Ramseyanity of the split, and the question about $\lambda(u'',v)$ is easy to answer using a variant of color detection (details omitted).
Altogether, when $h=\size{\Lambda}$, we have $\Oh(\size{\Lambda}\log\size{\Lambda})$ linear orders that determine $\lambda(u,x)$, as requested.

We remark that the mechanism presented above is very close to the treatment of monadic second-order queries on trees in the original work of Colcombet; see \cite[Lemma~3]{Colcombet07}.

\section{Preliminaries}\label{sec:prelims}

All graphs and posets that we consider are finite.
We use standard notation and terminology related to graphs.
We let $[n]=\set{1,\ldots,n}$ for $n\in\setN$.

In the rest of this section, we introduce all important or less standard terminology and notation and some preliminary results that will be used further in the paper---even if they were already introduced in \Cref{sec:overview}.

\subsection*{Relations}

A \emph{relation} on a set $I$ is a subset $\rho\subseteq I\times I$.
The \emph{composition} of relations $\rho,\rho'\subseteq I\times I$ is
\[\rho\rho' = \set[\big]{(i,k)\in I\times I\colon\text{there is }j\in I\text{ such that }(i,j)\in\rho\text{ and }(j,k)\in\rho'}{.}\]
The \emph{image} and the \emph{coimage} of a set $X\subseteq I$ under a relation $\rho\subseteq I\times I$ are, respectively,
\begin{align*}
\rho X &= \set[\big]{i\in I\colon\text{there is }j\in X\text{ such that }(i,j)\in\rho}\quad\text{and} \\
X\rho &= \set[\big]{j\in I\colon\text{there is }i\in X\text{ such that }(i,j)\in\rho}{.}
\end{align*}

A \emph{backward function} $\rho\colon I\from I$ is a relation $\rho\subseteq I\times I$ such that for every $j\in I$, there is a unique $i\in I$ with $(i,j)\in I$.
Note that for backward functions $\rho,\rho'\colon I\from I$, the relation $\rho\rho'$ defined above coincides with the ordinary functional composition of $\rho$ and $\rho'$.
Similarly, for a backward function $\rho\colon I\from I$ and a set $X\subseteq I$, the sets $\rho X$ and $X\rho$ defined above coincide with the ordinary functional image and coimage of $X$ under $\rho$, respectively.
Finally, for a backward function $\rho\colon I\from I$ and an element $j\in I$, $\rho j$ is the application of $\rho$ to $j$, i.e., the unique element $i\in I$ such that $(i,j)\in\rho$.

All operations presented above are associative in all configurations in which they can be applied.
The composition/image/coimage operation on pairs of relations/sets is defined coordinatewise.
For instance, for relations $\rho_1,\rho'_1,\rho_2,\rho'_2$, we define $(\rho_1,\rho_2)(\rho'_1,\rho'_2)=(\rho_1\rho'_1,\rho_2\rho'_2)$.

A \emph{directed graph} is a pair $(V,E)$ such that $E$ is a relation on $V$.

\subsection*{Posets}

A \emph{poset} is a pair $(X,\leq)$ such that $\leq$ is a reflexive, transitive, and antisymmetric relation on $X$.
Whenever $\leq$ is the order relation of a poset, we let $x<y$ denote that $x\leq y$ and $x\neq y$, and we let $x\parallel y$ denote that $x$ and $y$ are incomparable, i.e., $x\nleq y$ and $y\nleq x$.
The \emph{cover graph} of a poset $(X,\leq)$ is the (undirected) graph $(X,E)$ with edge set $E=\{xy\in E\colon x<y$ and there is no $z$ with $x<z<y\}$.

The \emph{dimension} of a poset $(X,\leq)$, denoted by $\dim(X,\leq)$, is the least $d\in\setN$ such that there exist $d$ linear orders $\preceq_1,\ldots,\preceq_d$ on $X$ satisfying
\[{\leq}={\preceq_1}\cap\cdots\cap{\preceq_d}.\]
An \emph{alternating cycle} in $(X,\leq)$ is a tuple of incomparable pairs $(x_1,y_1),\ldots,(x_m,y_m)$ such that $x_i\leq y_{i+1}$ for all $i\in[m]$ and these are the only comparabilities among $x_1,\ldots,x_m,y_1,\ldots,y_m$.
(Here and further on, we follow the convention that when speaking about alternating cycles, indices behave cyclically modulo the length of the cycle.
Thus, $x_m\leq y_1$.)
The following characterization of dimension through alternating cycles is well known; see~\cite{TM77}.

\begin{fact}\label{fc:dim-alt}
The dimension of a poset\/ $(X,\leq)$ is the least\/ $d\in\setN$ such that all (ordered) incomparable pairs in\/ $(X,\leq)$ can be colored with\/ $d$ colors so that there is no monochromatic alternating cycle.
\end{fact}

The \emph{Boolean dimension} of a poset $(X,\leq)$, denoted by $\bdim(X,\leq)$, is the least $b\in\setN$ such that there exist $b$ linear orders $\preceq_1,\ldots,\preceq_b$ on $X$ and a function $\decode\colon\set{0,1}^b\to\set{0,1}$, called the \emph{decoder}, such that the following is satisfied for every pair of distinct $x,y\in X$:
\[x\leq y \qquad\text{if and only if}\qquad \decode\bigl([x\preceq_1y],\ldots,[x\preceq_by]\bigr)=1.\]
Here, for an assertion $\psi$, $[\psi]$ denotes the value $1$ if $\psi$ holds and $0$ otherwise.
(Note that $\bdim(X,\leq)=0$ in case $\size{X}=1$.)
Clearly, the Boolean dimension of a poset is at most the dimension, because one can take the Boolean conjunction as the decoder.

\subsection*{Rooted trees}

A \emph{rooted tree} with \emph{root} $r$ is a triple $T=(V,E,r)$ such that $(V,E)$ is a tree and $r\in V$.
The elements of $V$ are called the \emph{nodes} of $T$.
The \emph{parent} of a node $u\neq r$ is the unique neighbor of $u$ on the path from $u$ to $r$ in $T$.
The \emph{children} of a node $u\in V$ are the neighbors of $u$ other than the parent of $u$.
The \emph{leaves} of $T$ are the nodes with no children.
The \emph{inner nodes} of $T$ are all the nodes except for leaves and the root.
An order $\downtoeq$ is defined on $V$ so that $u\downtoeq v$ if $u$ lies on the path from $v$ to $r$ in $T$.
When $u\downtoeq v$, we say that $u$ is an \emph{ancestor} of $v$ and $v$ is a \emph{descendant} of $u$ in $T$.
The \emph{lowest common ancestor} of two nodes $v$ and $w$ is the unique $\downtoeq$-maximal node $u$ such that $u\downtoeq v,w$.
By $u\downto v$ we mean that $u\downtoeq v$ and $u\neq v$.
Note that $(V,E)$ is the cover graph of the poset $(V,\downtoeq)$.

A \emph{subtree} of a rooted tree $T=(V,E,r)$ is a rooted tree $S=(V',E',r')$ such that $(V',E')$ is a subtree of $(V,E)$ (i.e., $V'\subseteq V$ and $E'\subseteq E$) and $r'$ is the $\downtoeq$-minimal node in $V'$.
Then the poset $(V',\downtoeq[S])$ is a subposet of $(V,\downtoeq)$.
A subtree $S$ of a rooted tree $T$ is \emph{proper} if for every node $u$ of $S$, either $u$ is a leaf of $S$ (but not necessarily of $T$) or all children of $u$ in $T$ are also nodes of $S$.
A \emph{binary tree} is a rooted tree in which the root and every inner node has exactly two children: the \emph{left child} and the \emph{right child}.
Note that a binary subtree of a binary tree is always proper.
A \emph{left} or \emph{right descendant} of a node $u$ in a binary tree is, respectively, a descendant of the left child or of the right child of $u$.

\subsection*{Splits}

Let $T$ be a rooted tree with inner node set $U$.
A \emph{split} of $T$ of \emph{order} $p\in\setN$ is a mapping $s\colon U\to[p]$.
In the context of a fixed split $s$, an \emph{$h$-node} is a node $u\in U$ with $s(u)=h$, two $h$-nodes $u\downto v$ are \emph{$h$-neighbors} if $s(w)\leq h$ for all nodes $w$ with $u\downto w\downto v$, and an \emph{$h$-subtree} is a maximal subtree of $T$ whose every inner node $u$ satisfies $s(u)\leq h$.
Such subtrees form a recursive partition of $T$, as follows:
\begin{itemize}
\item the unique $p$-subtree of $T$ is $T$ itself;
\item every $h$-subtree $S$ with $h\in[p]$ is partitioned by its $h$-nodes into $(h-1)$-subtrees so that every $h$-node in $S$ is a leaf of one $(h-1)$-subtree and the root of one $(h-1)$-subtree;
\item the $0$-subtrees are the minimal proper subtrees, i.e., the proper subtrees with no inner nodes.
\end{itemize}

Now, let $(\Lambda,\cdot)$ be a semigroup.
A \emph{$(\Lambda,\cdot)$-labeling} of $T$ is an assignment $\lambda$ of a label $\lambda(u,v)\in\Lambda$ to every pair of nodes $u\downto v$ with the property that
\[u\downto v\downto w\qquad\text{implies}\qquad\lambda(u,v)\cdot\lambda(v,w)=\lambda(u,w){.}\]
A split $s\colon U\to[p]$ of $T$ is \emph{forward Ramseyan} with respect to a $(\Lambda,\cdot)$-labeling $\lambda$ of $T$ if for every $h\in[p]$ and all pairs $u\downto v$ and $u'\downto v'$ of $h$-neighbors, we have
\[\lambda(u,v)\cdot\lambda(u',v')=\lambda(u,v){.}\]
We will rely on the following result due to Colcombet~\cite{Colcombet07}; see also \cite[Theorem~5.3]{Colcombet21}.

\begin{colcombet}
If\/ $(\Lambda,\cdot)$ is a finite semigroup, then every rooted tree\/ $T$ equipped with a\/ $(\Lambda,\cdot)$-labeling\/ $\lambda$ has a split of order\/ $\size{\Lambda}$ that is forward Ramseyan with respect to\/ $\lambda$.
\end{colcombet}

We remark that the assertion Colcombet's Theorem in \cite[Theorem~5.3]{Colcombet21} is slightly weaker, but the proof directly implies the statement above.
For the reader's convenience, we provide a self-contained proof of the theorem in \Cref{sec:colcombet}.

\subsection*{NLC-decompositions}

Instead of considering clique expressions and cliquewidth, we will be working with NLC-decompositions and NLC-width.
These notions differ slightly in terms of the set of allowed operations, but they are functionally equivalent.
The definition of NLC-decompositions and NLC-width for directed graphs that we are going to use is due to Gurski, Wanke, and Yilmaz~\cite{GWY16}.
In particular, as observed in \cite[Theorem~4.1]{GWY16}, cliquewidth and NLC-width of a directed graph differ by a multiplicative factor of at most $2$, hence a class of directed graphs has bounded cliquewidth if and only if it has bounded NLC-width.

Let $G=(X,E)$ be a directed graph.
An \emph{NLC-decomposition} of $G$ of \emph{width} $q$ is a sextuple $(T,Q,\eta,\rho,R,R')$ that consists of the following objects:
\begin{itemize}
\item a binary tree $T$ with leaf set $X$;
\item a \emph{label set} $Q$ of size $q$;
\item an \emph{initial label} $\initial{x}\in Q$ assigned to every leaf $x\in X$;
\item a \emph{relabeling} $\bagbb{u}{v}\colon Q\from Q$ defined for any two nodes $u\downto v$ so that
\[u\downto v\downto w\qquad\text{implies}\qquad\bagbb{u}{v}\bagbb{v}{w}=\bagbb{u}{w}{;}\]
by convention, for every node $u$ we set $\bagbb{u}{u}=\id_Q$ (the identity on $Q$);
\item \emph{status relations} $R(u),R'(u)\subseteq Q\times Q$ for every node $u$.
\end{itemize}
We require that the status relations at nodes of $T$ encode the relation $E$ in the following sense.
First, for every node $u$ and every leaf $x\in X$ with $u\downtoeq x$, define the \emph{label} $\bagbl{u}{x}\in Q$ of $x$ at $u$ as
\[\bagbl{u}{x}=\bagbb{u}{x}\initial{x}{.}\]
Then, consider any two distinct leaves $x,y\in X$, and let $u$ be their lowest common ancestor in $T$.
If $x$ is a left descendant and $y$ is a right descendant of $u$ in $T$, then we require that
\begin{align*}(x,y)&\in E\quad\text{if and only if}\quad(\bagbl{u}{x},\bagbl{u}{y})\in R(u){,}\quad\text{and} \\
(y,x)&\in E\quad\text{if and only if}\quad(\bagbl{u}{x},\bagbl{u}{y})\in R'(u){.}
\end{align*}

The \emph{NLC-width} of $G$ is the minimum width of an NLC-decomposition of $G$.
These definitions apply naturally to posets $(X,\leq)$ with the order relation $\leq$ playing the role of $E$.

\subsection*{Tree decompositions and minors}

Let $G=(X,E)$ be a graph.
A \emph{tree decomposition} of $G$ is a pair $(T,B)$ that consists of a rooted tree $T$ and an assignment of a \emph{bag} $B(u)\subseteq X$ to each node $u$ of $T$ which satisfies the following conditions:
\begin{itemize}
\item for every $x\in X$, the nodes $u$ in $T$ such that $x\in B(u)$ induce a nonempty subtree of $T$; and
\item for every edge $xy\in E$, there is a node $u$ of $T$ such that $x,y\in B(u)$.
\end{itemize}
The \emph{width} of a tree decomposition $(T,B)$ is $\max_u\size{B(u)}-1$.
The \emph{treewidth} of $G$ is the minimum width of a tree decomposition of $G$.

In our case, $G$ is always the cover graph of a poset $(X,\leq)$.
The conditions of tree decomposition then imply the following simple property.

\begin{fact}\label{fact:convexity}
Let\/ $u$, $v$, and\/ $w$ be nodes of\/ $T$ such that\/ $w$ lies on the path from\/ $u$ to\/ $v$ in\/ $T$.
If\/ $x\in B(u)$, $y\in B(v)$, and\/ $x\leq y$, then there is\/ $z\in B(w)$ such that\/ $x\leq z\leq y$.
\end{fact}

A bound on the treewidth of the cover graph of a poset implies a bound on the NLC-width of the poset itself.
This statement alone follows easily from the standard facts that classes of binary structures of bounded cliquewidth are closed under taking monadic second-order interpretations (see e.g.\ \cite[Theorem~1.39 and Corollary~1.43]{CE12} for a proof in the context of undirected graphs), that cliquewidth is bounded in terms of treewidth, and that obtaining the transitive closure can be expressed as a monadic second-order interpretation.
In \Cref{sec:kelly}, we use an explicit way of obtaining an NLC-decomposition of a poset from a tree decomposition of its cover graph, which provides the following direct bound.

\begin{fact}\label{fact:tw-NLC}
A poset with cover graph of treewidth\/ $t-1$ has NLC-width at most\/ $4^t$.
\end{fact}

A graph $H$ is a \emph{minor} of a graph $G$ if $H$ can be obtained from a subgraph of $G$ by contracting edges.
A \emph{model} of $H$ in $G$ is a collection $\mathcal{M}$ of vertex-disjoint connected subgraphs $M_x$ of $G$, one for each vertex $x$ of $H$, such that $M_x$ and $M_y$ are linked by an edge in $G$ for every edge $xy$ of $H$.
Note that $H$ is a minor of $G$ if and only if there is a model of $H$ in $G$.

\section{Dimension and NLC-width}\label{sec:dim-bnd}

In this section, we prove the following result, which directly implies \Cref{thm:dim-bounded-main}.

\begin{theorem}\label{thm:dim-bounded}
For every pair of integers\/ $q\geq 1$ and\/ $k\geq 2$, there is a constant\/ $d\in\setN$ such that every poset with NLC-width\/ $q$ and dimension greater than\/ $d$ contains the standard example\/ $S_k$ as a subposet.
\end{theorem}

We fix integers $q$ and $k$, a poset $(X,\leq)$, and an NLC-decomposition $(T,Q,\eta,\rho,R,R')$ of $(X,\leq)$ of width $q$.
Whenever speaking of comparability or incomparability, we mean the poset $(X,\leq)$.
We assume that $k\geq 3$ (which we can, as $S_2$ is a subposet of $S_3$) and $\size{X}\geq 2$.

For every node $u$ of $T$, we let
\[X(u)=\set{x\in X\colon u\downtoeq x}{.}\]
Observe that $u\downtoeq v$ entails $X(u)\supseteq X(v)$.
For every node $v$ of $T$ and every $x\in X\setminus X(v)$, we define subsets $\lablr{x}{v}$ and $\labgr{x}{v}$ of $Q$ as follows.
Let $u$ be the lowest common ancestor of $x$ and $v$ in $T$.
If $x$ is a left descendant and $v$ is a right descendant of $u$, then we define
\begin{align*}
\lablr{x}{v} &= \set[\big]{\gamma\in Q\colon(\bagbl{u}{x},\bagbb{u}{v}\gamma)\in R(u)}{,} \\
\labgr{x}{v} &= \set[\big]{\gamma\in Q\colon(\bagbl{u}{x},\bagbb{u}{v}\gamma)\in R'(u)}{.}
\end{align*}
Otherwise (i.e., if $x$ is a right descendant and $v$ is a left descendant of $u$), we define
\begin{align*}
\lablr{x}{v} &= \set[\big]{\gamma\in Q\colon(\bagbb{u}{v}\gamma,\bagbl{u}{x})\in R'(u)}{,} \\
\labgr{x}{v} &= \set[\big]{\gamma\in Q\colon(\bagbb{u}{v}\gamma,\bagbl{u}{x})\in R(u)}{.}
\end{align*}
These sets govern the interaction between $X(v)$ and $X\setminus X(v)$ in the following sense.

\begin{fact}\label{fact:L-sets}
Let\/ $v$ be a node of\/ $T$, let\/ $x\in X\setminus X(v)$, and let\/ $y\in X(v)$.
Then we have
\begin{align*}
x\leq y&\qquad\text{if and only if}\qquad\bagbl{v}{y}\in\lablr{x}{v}\text{,} \\
x\geq y&\qquad\text{if and only if}\qquad\bagbl{v}{y}\in\labgr{x}{v}\text{.}
\end{align*}
\end{fact}

\begin{proof}
Let $u$ be the lowest common ancestor of $x$ and $v$ (and thus of $x$ and $y$) in $T$.
Suppose that $x$ is a left descendant and $v$ is a right descendant of $u$.
By the definition of NLC-decomposition, we have $x\leq y$ if and only if $(\bagbl{u}{x},\bagbl{u}{y})\in R(u)$, which holds if and only if $\bagbl{v}{y}\in\lablr{x}{v}$, as $(\bagbl{u}{x},\bagbb{u}{v}\bagbl{v}{y})=(\bagbl{u}{x},\bagbl{u}{y})$.
The other cases are analogous.
\end{proof}

Let $U$ be the set of inner nodes of $T$.
A split $s\colon U\to[p]$ of $T$, where $p\in\setN$ is the order of~$s$, shall be called \emph{decent} if it is forward Ramseyan with respect to the labeling of $T$ with the relabeling function $\rho$.

We now introduce some terminology around decent splits.
Hence, fix for now some decent split $s\colon U\to[p]$.
In this context, for every $h\in[p]$, we define
\[E_h=\set{\bagbb{u}{v}\colon u\downto v\text{ are }h\text{-neighbors}}{.}\]
Since $s$ is decent, we have
\[\sigma\sigma'=\sigma\qquad\text{for all }\sigma,\sigma'\in E_h{.}\]
An \emph{$E_h$-class} is an equivalence class of the equivalence relation $\equiv_h$ on $Q$ defined by
\[\alpha\equiv_h\beta \qquad\text{if and only if}\qquad \sigma\alpha=\sigma\beta\text{ for all }\sigma\in E_h.\]
The following simple statement is a convenient tool for working with $E_h$-classes.

\begin{fact}\label{fact:E_h}
If\/ $C$ is an\/ $E_h$-class, $\gamma\in C$, and\/ $\tau\in E_h$, then\/ $\tau\gamma\in C$.
\end{fact}

\begin{proof}
We have $\sigma(\tau\gamma)=(\sigma\tau)\gamma=\sigma\gamma$ for all $\sigma\in E_h$.
Thus $\tau\gamma\equiv_h\gamma$, so $\tau\gamma\in C$.
\end{proof}

For $h\in[p]$, an \emph{$h$-chain} of \emph{length} $\ell$ is a sequence of $\ell+1$ mutual $h$-neighbors $u_0\downto\cdots\downto u_\ell$.
An \emph{$h$-cross} is a quintuple $(u,v,x,y,\sigma)$ such that
\begin{itemize}
\item $u\downto v$ are $h$-neighbors;
\item $x,y\in X(u)\setminus X(v)$ are incomparable; and
\item $\sigma\in E_h$, $\sigma\bagbl{u}{x}\in\labgr{y}{v}$, and $\sigma\bagbl{u}{y}\in\lablr{x}{v}$.
\end{itemize}
Note that again, these objects are defined in the context of a decent split $s$.
The next lemma is a crucial step in the proof: if $(X,\leq)$ has large dimension, then this forces the existence of a long $h$-chain and an $h$-cross, for some $h\in[p]$.

\begin{lemma}\label{dimension}
Let\/ $s$ be a decent split of order\/ $p$.
Then for every\/ $\ell\in\setN$, there exists a constant\/ $d\in\setN$, depending only on\/ $\ell$, $p$, and\/ $q$, such that at least one of the following conditions holds:
\begin{itemize}
\item $\dim(X,\leq)\leq d$; or
\item for some\/ $h\in[p]$, there exist an\/ $h$-chain of length\/ $\ell$ and an\/ $h$-cross.
\end{itemize}
\end{lemma}

\begin{proof}
We start by introducing some more terminology.
Let $S$ be a binary subtree of $T$.
A \emph{cluster} of $S$ is a set of the form $X(v)$ for some leaf $v$ of~$S$.
The clusters of $S$ form a partition of $X(r)$ where $r$ is the root of~$S$.
An \emph{inner incomparable pair} of $S$ is an incomparable pair $(x,y)$ such that $x$ and $y$ lie in the same cluster of~$S$.
An \emph{outer incomparable pair} of $S$ is an incomparable pair $(x,y)$ such that $x$ and $y$ lie in different clusters of~$S$.
An \emph{inner alternating cycle} of $S$ is a sequence $(x_1,y_1),\ldots,(x_m,y_m)$ of inner incomparable pairs of $S$ such that $x_i\leq y_{i+1}$ for $i\in[m]$, other pairs of $x_1,y_1,\ldots,x_m,y_m$ are incomparable, and $x_1,y_1,\ldots,x_m,y_m$ do not all lie in the same cluster.\footnote{The reader might wonder why we exclude the case where $x_1,y_1,\ldots,x_m,y_m$ are all in the same cluster. As will become apparent in the proof, it turns out that $S$ is not the right ``place'' to consider the latter alternating cycles.}
An \emph{outer alternating cycle} of $S$ is a sequence $(x_1,y_1),\ldots,(x_m,y_m)$ of outer incomparable pairs of $S$ such that for each $i\in[m]$, we have $x_i\leq y_{i+1}$ or $x_i$ and $y_{i+1}$ lie in the same cluster of $S$ (in particular, we allow $x_i=y_{i+1}$), and other pairs of $x_1,y_1,\ldots,x_m,y_m$ are incomparable.
(Let us point out that an outer alternating cycle of $S$ is not necessarily an alternating cycle of the poset in the usual sense, since there is no condition on how $x_i$ and $y_{i+1}$ compare when they are in the same cluster of $S$.)
For $\type\in\{$inner$,{}$outer$\}$, a \emph{proper coloring} of the $\type$ incomparable pairs of $S$ is a coloring that makes no $\type$ alternating cycle of $S$ monochromatic.

We prove the following by induction on $h$: for every $h\in\set{0,1,\ldots,p}$, there are constants $d^\inner_h$ and $d^\outer_h$ such that at least one of the following two conditions holds:
\begin{enumerate}
\item\label{o:coloring} for every $h$-subtree $S$ and every $\type\in\{$inner$,{}$outer$\}$, there is a proper $d^\type_h$-coloring of the $\type$ incomparable pairs of $S$; or
\item\label{o:chain} for some $h'\in[h]$, there exist an $h'$-chain of length $\ell$ and an $h'$-cross.
\end{enumerate}
The lemma then follows by taking $d=d^\outer_p$, as, by \Cref{fc:dim-alt}, $\dim(X,\leq)$ is the minimum number of colors in a proper coloring of the (outer) incomparable pairs of $T$.

We settle the base case $h=0$ by setting $d^\inner_0=2$ and $d^\outer_0=2q$ and showing that \ref{o:coloring} holds, as follows.
Let $S$ be a $0$-subtree with root $u$ and leaves $v$ and $w$, where $v$ is the left child and $w$ is the right child of $u$.
A $2$-coloring of the inner incomparable pairs of $S$ that uses one color on the incomparable pairs in $X(v)\times X(v)$ and the other on those in $X(w)\times X(w)$ is proper, as every inner alternating cycle of $S$ (by definition) contains incomparable pairs from both of these sets.
As for the outer incomparable pairs of $S$, color them with $2q$ colors using one color on the incomparable pairs $(x,y)\in X(v)\times X(w)$ for each value of $\bagbl{u}{x}\in Q$ and one color on the incomparable pairs $(x,y)\in X(w)\times X(v)$ for each value of $\bagbl{u}{x}\in Q$.
Suppose that $(x_1,y_1),\ldots,(x_m,y_m)$ is a monochromatic outer alternating cycle of $S$ in this coloring.
Thus, the pairs $(x_i,y_i)$ either all belong to $X(v)\times X(w)$ or all belong to $X(w)\times X(v)$.
Consequently, since $x_1\leq y_2$, the definition of the order $\leq$ through the status relation $R(u)$ or $R'(u)$ (respectively) yields $x_2\leq y_2$, which is a contradiction.
Hence, the coloring of the outer alternating cycles of $S$ is proper.
This completes the base case of the induction.

We proceed to the induction step for $h\in[p]$.
By the induction hypothesis, \ref{o:coloring} or \ref{o:chain} holds for $h-1$.
The latter case immediately yields \ref{o:chain} for $h$, so assume the former.
In other words, for every $(h-1)$-subtree $S$ and every $\type\in\{$inner$,{}$outer$\}$, we can find a proper $d^\type_{h-1}$-coloring of the $\type$ incomparable pairs of $S$; call it $\phi^\type\langle S\rangle$.
We assume that all colorings $\phi^\inner\langle S\rangle$ share the same palette of $d^\inner_{h-1}$ colors and all colorings $\phi^\outer\langle S\rangle$ share the same palette of $d^\outer_{h-1}$ colors.

Consider an arbitrary $h$-subtree $S^*$.
Let $r^0$ be the root of $S^*$ and $S^0$ be the unique $(h-1)$-subtree of $S^*$ with root $r^0$.
The \emph{level} of an $(h-1)$-subtree $S$ of $S^*$ with root $r$ is the number of $h$-nodes on the path from $r^0$ to $r$ in $T$.
Thus $S^0$ is the unique $(h-1)$-subtree of $S^*$ at level $0$.
Note that an $(h-1)$-subtree of $S^*$ at level $\ell+1$ yields an $h$-chain of length $\ell$ formed by the $h$-nodes that witness the level.
For each $\type\in\{$inner$,{}$outer$\}$, we provide two different ways of coloring the $\type$ incomparable pairs of $S^*$ depending on whether or not there exists an $(h-1)$-subtree of $S^*$ at level greater than $\ell$, and we set $d^\type_h$ to be the maximum number of colors used in the respective coloring:
\[d^\type_h=d^\type_{h-1}\cdot\max\set[\Bigg]{\sum_{i=0}^\ell\bigl(d^\inner_{h-1}\bigr)^i,\;1+3q^3\bigl(d^\inner_{h-1}\bigr)^2}{.}\]

\begin{case}
The level of every $(h-1)$-subtree of $S^*$ is at most $\ell$.
\end{case}

Consider an arbitrary $\type$ incomparable pair $(x,y)$ of $S^*$.
Let $S^0,\ldots,S^j$ be the unique $(h-1)$-subtrees at levels $0,\ldots,j$, respectively, such that $(x,y)$ is an incomparable pair of $S^0,\ldots,S^{j}$, where $j$ is maximum (at most $\ell$).
Thus $(x,y)$ is an inner incomparable pair of $S^0,\ldots,S^{j-1}$ and a $\type$ incomparable pair of $S^j$.
The aimed coloring of $\type$ incomparable pairs of $S^*$ assigns the following color to $(x,y)$:
\[\bigl(j,\:\phi^\inner\langle S^0\rangle(x,y),\:\ldots,\:\phi^\inner\langle S^{j-1}\rangle(x,y),\:\phi^\type\langle S^j\rangle(x,y)\bigr){.}\]

We claim that this coloring is proper.
Suppose for the sake of contradiction that the resulting coloring yields a monochromatic $\type$ alternating cycle $C$ of $S^*$.
Let $(j,c^0,\ldots,c^j)$ be the common color of the members of $C$.
Let $S$ be the unique $(h-1)$-subtree of $S^*$ at maximum level such that the members of $C$ are incomparable pairs of $S$, and let $i$ be the level of $S$.
It follows that $i\leq j$ and the $(h-1)$-subtrees $S^0,\ldots,S^i$ (where $S^i=S$) used in the color assignment are common for all members of~$C$.
Thus, if $i<j$, then $C$ is an inner alternating cycle of $S$ that is monochromatic in $\phi^\inner\langle S\rangle$ with color $c^i$.
(Note that not all incomparable pairs of $C$ are in the same cluster of $S$ by our choice of $S$.)
If $i=j$, then $C$ is a $\type$ alternating cycle of $S$ that is monochromatic in $\phi^\type\langle S\rangle$ with color $c^j$.
This is a contradiction.

\begin{case}
There is an $(h-1)$-subtree of $S^*$ at level greater than $\ell$, and consequently there is an $h$-chain of length $\ell$.
\end{case}

Suppose henceforth that there is no $h$-cross.
We aim at constructing a proper $d^\type_h$-coloring of the $\type$ incomparable pairs of $S^*$.
Thus, consider an arbitrary $\type$ incomparable pair $(x,y)$ of $S^*$.
Let $S$ be the unique $(h-1)$-subtree with maximum level such that $(x,y)$ is an incomparable pair of $S$, let $r$ be the root of $S$, and let $j$ be the level of $S$.
If $j\geq 1$, then let $S'$ be the unique $(h-1)$-subtree with level $j-1$ such that $(x,y)$ is an incomparable pair of $S'$, and let $r'$ be the root of $S'$.
Thus, $(x,y)$ is a $\type$ incomparable pair of $S$ and (if $j\geq 1$) an inner incomparable pair of $S^0$ and $S'$.
The aimed coloring of $\type$ incomparable pairs of $S^*$ assigns the following color to $(x,y)$:
\begin{alignat*}{3}
\bigl(&0,\:&&\phi^\type\langle S^0\rangle(x,y)\bigr) && \quad\text{if }j=0{,} \\
\bigl(&j',\:&&\phi^\inner\langle S^0\rangle(x,y),\:\phi^\type\langle S\rangle(x,y),\:\phi^\inner\langle S'\rangle(x,y),\:\bagbl{r}{x},\:\bagbl{r}{y},\:\bagbl{r'}{x}\bigr) && \quad\text{if }j\geq 1{,}
\end{alignat*}
where $j'\in\set{1,2,3}$ is such that $j'\equiv j\pmod{3}$.

We claim that this coloring is proper.
Suppose for the sake of contradiction that the resulting coloring yields a monochromatic $\type$ alternating cycle $C$ of $S^*$ comprised of incomparable pairs $(x_1,y_1),\ldots,(x_m,y_m)$.
If their common color is $(0,c^0)$, then $C$ is a $\type$ alternating cycle of $S^0$ that is monochromatic in $\phi^\type\langle S^0\rangle$ with color $c^0$, which is a contradiction.
Thus, the common color of $(x_1,y_1),\ldots,(x_m,y_m)$ has $1$, $2$, or $3$ on the first coordinate.
Let $(j',c^0,c,c',\alpha,\beta,\alpha')$ be that common color, where $j'\in\set{1,2,3}$.

For each $i\in[m]$, let $S_i$ be the unique $(h-1)$-subtree at maximum level such that $(x_i,y_i)$ is an incomparable pair of $S_i$, let $r_i$ be the root of $S_i$, and let $j_i$ be the level of $S_i$.
It follows that $j_i\geq 1$ and $j_i\equiv j'\pmod{3}$, for all $i\in[m]$.
Let $\bar S$ be the unique $(h-1)$-subtree at maximum level such that $(x_1,y_1),\ldots,(x_m,y_m)$ are incomparable pairs of $\bar S$, let $\bar r$ be the root of $\bar S$, and let $\bar j$ be the level of $\bar S$.
For each $i\in[m]$, we have $\bar j\leq j_i$, and $(x_i,y_i)$ is an inner incomparable pair (if $\bar j<j_i$) or a $\type$ incomparable pair (if $\bar j=j_i$) of $\bar S$.

If $\bar j=0$, so that $\bar S=S^0$, then $C$ is an inner alternating cycle of $S^0$ that is monochromatic in $\phi^\inner\langle S^0\rangle$ with color $c^0$, which is a contradiction.
(Note that incomparable pairs of $C$ cannot all be in the same cluster of $S^0$ by our choice of $\bar S$.)
Thus $\bar j\geq 1$, so $\bar S\neq S^0$ and $\bar r$ is an $h$-node.

Let $A$ and $B$ be the $E_h$-classes such that $\alpha\in A$ and $\beta\in B$.

\begin{claim}\label{claim1}
For every $h$-node $u$ of $S^*$ and every $i\in[m]$, if $x_i\in X(u)$ then $\bagbl{u}{x_i}\in A$, and if $y_i\in X(u)$ then $\bagbl{u}{y_i}\in B$.
\end{claim}

\begin{proof}
For $u=r_i$, we have $\bagbl{r_i}{x_i}=\alpha\in A$.
If $u\neq r_i$, then $u$ and $r_i$ are $h$-neighbors either with $u\downto r_i$, which by \Cref{fact:E_h} implies $\bagbl{u}{x_i}=\bagbb{u}{r_i}\alpha\in A$ due to $\bagbb{u}{r_i}\in E_h$, or with $r_i\downto u$, which implies $\alpha=\bagbb{r_i}{u}\bagbl{u}{x_i}$, and thus, by \Cref{fact:E_h}, $\bagbl{u}{x_i}\in A$ due to $\bagbb{r_i}{u}\in E_h$.
The proof for $y_i$ is analogous.
\end{proof}

\begin{claim}\label{claim2}
For any $h$-nodes $u$ and $v$ of $S^*$ such that $\bar r\downto u\downto v$ are $h$-neighbors, there is no $i\in[m]$ such that $x_i,y_i\in X(v)$ and $x_{i-1},y_{i+1}\notin X(u)$.
\end{claim}

\begin{proof}
Let $\sigma=\bagbb{u}{v}$; we have $\sigma\in E_h$.
\Cref{claim1} yields that $\bagbl{\bar r}{x_{i-1}},\bagbl{v}{x_i}\in A$ and $\bagbl{\bar r}{y_{i+1}},\bagbl{v}{y_i}\in B$.
This and the fact that $A$ and $B$ are $E_h$-classes imply that
\[\sigma\bagbl{\bar r}{x_{i-1}}=\sigma\bagbl{v}{x_i}=\bagbl{u}{x_i}\in\labgr{y_{i+1}}{u}{,}\]
where the membership follows from \Cref{fact:L-sets}.
Similarly, we have
\[\sigma\bagbl{\bar r}{y_{i+1}}=\sigma\bagbl{v}{y_i}=\bagbl{u}{y_i}\in\lablr{x_{i-1}}{u}{.}\]
Since $x_{i-1}\parallel y_{i+1}$, we conclude that $(\bar r,u,x_{i-1},y_{i+1},\sigma)$ is an $h$-cross.
This contradicts the assumption that there is no $h$-cross.
\end{proof}

\begin{claim}\label{claim3}
For every $i\in[m]$, if $x_i$ and $y_{i+1}$ lie in different clusters of $\bar S$, then $j_i,j_{i+1}\leq\bar j+2$.
\end{claim}

\begin{proof}
Suppose for the sake of contradiction that $j_i\geq\bar j+3$.
Then there are $h$-neighbors $u\downto v\downto w$ in $S^*$ such that $u$ is a leaf of $\bar S$ and $x_i,y_i\in X(w)$ while $y_{i+1}\notin X(u)$.
If $x_{i-1}\in X(v)$, then $\bagbl{v}{x_{i-1}},\bagbl{v}{x_i}\in A$ by \Cref{claim1}, so
\[\bagbl{u}{x_{i-1}}=\bagbb{u}{v}\bagbl{v}{x_{i-1}}=\bagbb{u}{v}\bagbl{v}{x_i}=\bagbl{u}{x_i},\]
where the second equality is due to $A$ being an $E_h$-class.
By \Cref{fact:L-sets}, $x_i\leq y_{i+1}$ entails $\bagbl{u}{x_i}\in\labgr{y_{i+1}}{u}$ while $x_{i-1}\parallel y_{i+1}$ entails $\bagbl{u}{x_{i-1}}\notin\labgr{y_{i+1}}{u}$, which contradicts the equality above.
Thus $x_{i-1}\notin X(v)$, which contradicts \Cref{claim2} applied to $h$-neighbors $\bar r\downto v\downto w$.
The proof that $j_{i+1}\leq\bar j+2$ is analogous.
\end{proof}

The rest of the proof splits into two cases, the first of which can occur only when $\type={}$outer.

\begin{subcase}
At least one of $(x_1,y_1),\ldots,(x_m,y_m)$ is an outer incomparable pair of $\bar S$.
\end{subcase}

For those indices $i\in[m]$ for which $(x_i,y_i)$ is an outer incomparable pair of $\bar S$, we have $\bar S=S_i$ and $\bar j=j_i\equiv j'\pmod{3}$.
On the other hand, if $(x_i,y_i)$ is an inner incomparable pair of $\bar S$, then $j_i>\bar j$, which due to $j_i\equiv j'\pmod{3}$ implies $j_i\geq \bar j+3$.
In either case, by \Cref{claim3}, $x_{i-1}$, $x_i$, $y_i$, and $y_{i+1}$ all lie in the same cluster of $\bar S$.
Consequently, the outer incomparable pairs of $\bar S$ among $(x_1,y_1),\ldots,(x_m,y_m)$ form an outer alternating cycle $\bar C$ of $\bar S$.
The cycle $\bar C$ is monochromatic in $\phi^\outer\langle\bar S\rangle$ with color $c$, which is a contradiction.

\begin{subcase}
All $(x_1,y_1),\ldots,(x_m,y_m)$ are inner incomparable pairs of $\bar S$.
\end{subcase}

Let $j\in\set{\bar j,\bar j+1,\bar j+2}$ be such that $j\equiv j'\pmod{3}$.
We claim that
\[j_1=\cdots=j_m=j\in\set{\bar j,\bar j+1}{.}\]
For the proof, consider any maximal contiguous part $(x_i,y_i),\ldots,(x_{i'},y_{i'})$ of $C$ that lies in a single cluster $X(u)$ of $\bar S$, so that $x_{i-1}$ and $y_{i'+1}$ lie in clusters other than $X(u)$.
(Note here that by the choice of $\bar S$, not all $(x_1,y_1),\ldots,(x_m,y_m)$ lie in a single cluster of $\bar S$.)
\Cref{claim3} asserts that $j_i,j_{i'}\leq\bar j+2$.
As $j_i\equiv j_{i'}\equiv j'\pmod{3}$, it follows that $j_i=j_{i'}=j$.
If $j=\bar j$ (which is possible only when $\type={}$inner), then $X(u)$ is a cluster of $S^*$, so $j_i=\cdots=j_{i'}=\bar j$.
If $j=\bar j+1$, then we have $\bagbl{u}{x_i}=\alpha=\bagbl{u}{x_{i'}}$.
By \Cref{fact:L-sets}, this and $x_{i'}\leq y_{i'+1}$ imply $\bagbl{u}{x_i}\in\labgr{y_{i'+1}}{u}$, which entails $x_i\leq y_{i'+1}$, which is possible only when $i=i'$.
Finally, if $j=\bar j+2$, then the fact that $\bagbl{u}{x_i}=\alpha'=\bagbl{u}{x_{i'}}\in\labgr{y_{i'+1}}{u}$ yields $i=i'$ again, but now \Cref{claim2} applied to $h$-neighbors $\bar r\downto u\downto r_i$ yields a contradiction.
So we argued in all cases that we have $j_i=\cdots=j_{i'}=j\in\set{\bar j,\bar j+1}$.
Applying this reasoning to every contiguous part of $C$ as above yields the claim.

Now, if $j_1=\cdots=j_m=\bar j$ (and $\type={}$inner), then $C$ is an inner alternating cycle of $\bar S$ that is monochromatic in $\phi^\inner\langle\bar S\rangle$ with color $c$, which is a contradiction.
If $j_1=\cdots=j_m=\bar j+1$, then $C$ is an inner alternating cycle of $\bar S$ that is monochromatic in $\phi^\inner\langle\bar S\rangle$ with color $c'$, which is a contradiction.
\end{proof}

Our goal now is to construct a large standard example using a large $h$-chain and an $h$-cross.
The following lemma exhibits a structure that allows us to find such a standard example.

\begin{lemma}\label{h-cross}
There is a decent split of order\/ $p=q^q\cdot 2^{(q+q^2)4^q}$ with the property that for every\/ $h\in[p]$, if there is an\/ $h$-cross, then there are\/ $E_h$-classes\/ $A$ and\/ $B$ such that the following holds for every\/ $h$-chain\/ $u_0\downto\cdots\downto u_k$.
Let\/ $\tau_j=\bagbb{u_{j-1}}{u_j}$, and let\/ $\alpha_j\in A$ and\/ $\beta_j\in B$ be such that\/ $\tau_jA=\set{\alpha_j}$ and\/ $\tau_jB=\set{\beta_j}$, for each\/ $j\in[k]$.
Then there are incomparable pairs\/ $x_j,y_j\in X(u_{j-1})\setminus X(u_j)$ for all\/ $j\in[k]$ such that the following two conditions are satisfied:
\begin{enumerate}
\item for every\/ $j\in[k]$, we have
\[\bagbl{u_{j-1}}{x_j}=\alpha_j\text{,} \quad \bagbl{u_{j-1}}{y_j}=\beta_j\text{,} \quad A\subseteq\labgr{y_j}{u_j}\text{,} \quad\text{and}\quad B\subseteq\lablr{x_j}{u_j}\text{;}\]
\item for every\/ $j\in\set{2,\ldots,k}$, we have
\[\labgr{y_j}{u_j}=\labgr{y_{j-1}}{u_{j-1}}\tau_j \quad\text{and}\quad \lablr{x_j}{u_j}=\lablr{x_{j-1}}{u_{j-1}}\tau_j\text{.}\]
\end{enumerate}
\end{lemma}

\begin{proof}
Let $\Lambda$ be the set of all triples $(\rho,\Phi,\Psi)$ such that
\[\rho\colon Q\from Q{,} \qquad \Phi\subseteq Q\times 2^Q\times 2^Q{,} \qquad \Psi\subseteq Q\times Q\times 2^Q\times 2^Q{.}\]
We have $\size{\Lambda}=q^q\cdot 2^{(q+q^2)4^q}=p$.
For every pair of nodes $u\downto v$, we define
\[\lambda(u,v)=(\bagbb{u}{v},\Phi(u,v),\Psi(u,v))\in\Lambda\]
where
\begin{align*}
\Phi(u,v) &= \set[\big]{\bigl(\bagbl{u}{x},\lablr{x}{v},\labgr{x}{v}\bigr)\colon x\in X(u)\setminus X(v)}{,}\\
\Psi(u,v) &= \set[\big]{\bigl(\bagbl{u}{x},\bagbl{u}{y},\lablr{x}{v},\labgr{y}{v}\bigr)\colon x,y\in X(u)\setminus X(v)\text{ such that }x\parallel y}{.}
\end{align*}
We define a binary operation $\cdot$ on $\Lambda$ in the unique possible way that makes $(\Lambda,\cdot)$ a semigroup and $\lambda$ a $(\Lambda,\cdot)$-labeling of $T$.
Namely,
\[(\rho_1,\Phi_1,\Psi_1)\cdot(\rho_2,\Phi_2,\Psi_2)=(\rho_1\rho_2,\Phi',\Psi')\]
where
\begin{align*}
\Phi' ={}& \set[\big]{(\alpha,A\rho_2,B\rho_2)\colon(\alpha,A,B)\in\Phi_1}\cup\set[\big]{(\rho_1\alpha,A,B)\colon(\alpha,A,B)\in\Phi_2}{,} \\
\Psi' ={}& \set[\big]{(\alpha,\beta,A\rho_2,B\rho_2)\colon(\alpha,\beta,A,B)\in\Psi_1}\cup\set[\big]{(\rho_1\alpha,\rho_1\beta,A,B)\colon(\alpha,\beta,A,B)\in\Psi_2} \\[-\jot]
&\quad \cup\set[\big]{(\alpha_1,\rho_1\alpha_2,A_1\rho_2,B_2)\colon(\alpha_1,A_1,B_1)\in\Phi_1,\:(\alpha_2,A_2,B_2)\in\Phi_2,\:\alpha_2\notin A_1\cup B_1} \\[-\jot]
&\quad \cup\set[\big]{(\rho_1\alpha_2,\alpha_1,A_2,B_1\rho_2)\colon(\alpha_1,A_1,B_1)\in\Phi_1,\:(\alpha_2,A_2,B_2)\in\Phi_2,\:\alpha_2\notin A_1\cup B_1}{.}
\end{align*}
It is straightforward but tedious to verify that $\cdot$ is indeed an associative operation, while the following claim asserts that $\lambda$ is indeed a $(\Lambda,\cdot)$-labeling of $T$.

\begin{claim}
If\/ $u\downto v\downto w$, then\/ $\lambda(u,w)=\lambda(u,v)\cdot\lambda(v,w)$.
\end{claim}

\begin{proof}
Fix a triple of nodes $u\downto v\downto w$.
Let
\begin{alignat*}{4}
(\rho_1,\Phi_1,\Psi_1) &= (\rho(u,v),{}&&\Phi(u,v),{}&&\Psi(u,v)) &&= \lambda(u,v){,} \\
(\rho_2,\Phi_2,\Psi_2) &= (\rho(v,w),{}&&\Phi(v,w),{}&&\Psi(v,w)) &&= \lambda(v,w){.}
\end{alignat*}
We need to show that
\[\lambda(u,w) = (\rho(u,w),\Phi(u,w),\Psi(u,w)) = (\rho_1\rho_2,\Phi',\Psi'){,}\]
where $\Phi'$ and $\Psi'$ are defined above.
It is clear that $\rho(u,w)=\rho_1\rho_2$.
Observe that
\begin{enumerate}
\item\label{item:lambda-1} if $x\in X(u)\setminus X(v)$, then, by the definition of $\lablr{x}{w}$, we have $\gamma\in\lablr{x}{w}$ if and only if $\rho_2\gamma=\bagbb{v}{w}\gamma\in\lablr{x}{v}$, which implies (using the coimage notation for relations) that $\lablr{x}{w}=\lablr{x}{v}\rho_2$, and analogously $\labgr{x}{w}=\labgr{x}{v}\rho_2$;
\item\label{item:lambda-2} if $x\in X(v)\setminus X(w)$, then $\bagbl{u}{x}=\bagbb{u}{v}\bagbl{v}{x}=\rho_1\bagbl{v}{x}$;
\item\label{item:lambda-3} if $x\in X(u)\setminus X(v)$ and $y\in X(v)\setminus X(w)$, then, by \Cref{fact:L-sets}, we have $x\parallel y$ if and only if $\bagbl{v}{y}\notin\lablr{x}{v}\cup\labgr{x}{v}$.
\end{enumerate}

Now, to see that $\Phi(u,w)=\Phi'$, observe that the two terms in the formula for $\Phi'$ correspond to two cases of triples $\bigl(\bagbl{u}{x},\lablr{x}{w},\labgr{x}{w}\bigr)\in\Phi(u,w)$: the first that $x\in X(u)\setminus X(v)$ and the second that $x\in X(v)\setminus X(w)$.
The specific form of the triples follows from \ref{item:lambda-1} and \ref{item:lambda-2} above.

Similarly, to see that $\Psi(u,w)=\Psi'$, observe that the four terms in the formula for $\Psi'$ correspond to four cases of quadruples $\bigl(\bagbl{u}{x},\bagbl{u}{y},\lablr{x}{w},\labgr{y}{w}\bigr)\in\Psi(u,w)$: the first that $x,y\in X(u)\setminus X(v)$, the second that $x,y\in X(v)\setminus X(w)$, the third that $x\in X(u)\setminus X(v)$ and $y\in X(v)\setminus X(w)$, and the fourth that $y\in X(u)\setminus X(v)$ and $x\in X(v)\setminus X(w)$.
The specific form of the quadruples follows from \ref{item:lambda-1} and \ref{item:lambda-2} applied to $x$ or $y$ accordingly, and the specification predicate used in the last two terms follows from \ref{item:lambda-3}.
\end{proof}

By Colcombet's Theorem, there is a split $s\colon U\to[p]$ of $T$ that is forward Ramseyan with respect to $\lambda$.
Since the relabeling function $\rho$ is on the first coordinate of $\lambda$, the split $s$ is decent.
We fix $s$ for the remainder of the proof; in particular, all considered crosses or chains are with respect to $s$.

For every $h\in[p]$, let
\[F_h=\set{\lambda(u,v)\colon u\downto v\text{ are }h\text{-neighbors in }s}{.}\]
Note that the set $E_h$ used for the definition of an $h$-cross is just the projection of $F_h$ on the first coordinate.
Since $s$ is forward Ramseyan, we have
\[\nu\cdot\nu'=\nu \qquad\text{for all }\nu,\nu'\in F_h{.}\]

Now, let $h\in[p]$, and let $(u,v,x,y,\sigma)$ be an $h$-cross.
Let $\alpha=\bagbl{u}{x}$ and $\beta=\bagbl{u}{y}$.
Thus $\sigma\alpha\in\labgr{y}{v}$ and $\sigma\beta\in\lablr{x}{v}$.
Let $A$ and $B$ be the $E_h$-classes such that $\alpha\in A$ and $\beta\in B$.
Since $A$ and $B$ are $E_h$-classes and $\sigma\in E_h$, we have $\sigma A=\set{\sigma\alpha}\subseteq\labgr{y}{v}$ and $\sigma B=\set{\sigma\beta}\subseteq\lablr{x}{v}$.
This implies
\begin{equation}\label{eq:AB}
A\subseteq\labgr{y}{v}\sigma \qquad\text{and}\qquad B\subseteq\lablr{x}{v}\sigma{.}
\end{equation}
Since $u\downto v$ are $h$-neighbors, we have $\lambda(u,v)=(\bagbb{u}{v},\Phi(u,v),\Psi(u,v))\in F_h$ where
\begin{equation}\label{eq:lambda}
\bigl(\alpha,\:\beta,\:\lablr{x}{v},\:\labgr{y}{v}\bigr)\in\Psi(u,v){.}
\end{equation}

Consider an $h$-chain $u_0\downto\cdots\downto u_k$.
For each $j\in[k]$, let
\[\tau_j=\bagbb{u_{j-1}}{u_j}{,}\quad\Phi_j=\Phi(u_{j-1},u_j){,}\quad\Psi_j=\Psi(u_{j-1},u_j){,}\quad\nu_j=(\tau_j,\Phi_j,\Psi_j)=\lambda(u_{j-1},u_j){,}\]
and let\/ $\alpha_j\in A$ and\/ $\beta_j\in B$ be such that\/ $\tau_jA=\set{\alpha_j}$ and\/ $\tau_jB=\set{\beta_j}$.
We construct sets $A_1,B_1,\ldots,A_k,B_k\subseteq Q$ with the following two properties:
\begin{gather}
(\alpha_j,\beta_j,B_j,A_j)\in\Psi_j{,} \quad A\subseteq A_j{,} \quad\text{and}\quad B\subseteq B_j{,} \quad\text{for all }j\in[k]{;} \label{eq:h-cross-1} \\
A_j=A_{j-1}\tau_j \quad\text{and}\quad B_j=B_{j-1}\tau_j{,} \quad\text{for all }j\in\set{2,\ldots,k}{.} \label{eq:h-cross-2}
\end{gather}
We proceed by induction on $j\in[k]$.
Let $A_1=\labgr{y}{v}\sigma$ and $B_1=\lablr{x}{v}\sigma$, so that $A\subseteq A_1$ and $B\subseteq B_1$ by~\eqref{eq:AB}.
Since $\sigma\in E_h$, there exists $\mu\in F_h$ with first coordinate $\sigma$, say $\mu=(\sigma,\Phi_\mu,\Psi_\mu)$.
Since $\nu_1,\lambda(u,v),\mu\in F_h$, we have
\[\lambda(u,v)\cdot\mu=\lambda(u,v) \qquad\text{and}\qquad \nu_1\cdot\lambda(u,v)=\nu_1{.}\]
Property~\eqref{eq:lambda}, the fact that $\lambda(u,v)\cdot\mu=\lambda(u,v)$, and the definition of $\cdot$ (specifically, the first term in the formula for $\Psi'$) imply
\[\bigl(\alpha,\:\beta,\:\lablr{x}{v}\sigma,\:\labgr{y}{v}\sigma\bigr)\in\Psi(u,v){.}\]
This, the fact that $\nu_1\cdot\lambda(u,v)=\nu_1$, and the definition of $\cdot$ (specifically, the second term in the formula for $\Psi'$) imply
\[(\alpha_1,\beta_1,B_1,A_1)=\bigl(\tau_1\alpha,\:\tau_1\beta,\:\lablr{x}{v}\sigma,\:\labgr{y}{v}\sigma\bigr)\in\Psi_1{.}\]
This shows \eqref{eq:h-cross-1} for $j=1$.
Now, let $j\in\set{2,\ldots,k}$, and assume that both \eqref{eq:h-cross-1} and \eqref{eq:h-cross-2} hold for $j-1$.
Let $A_j$ and $B_j$ be defined as in \eqref{eq:h-cross-2}.
By an argument similar to the above, the fact that $\nu_{j-1}\cdot\nu_j=\nu_{j-1}$ and $\nu_j\cdot\nu_{j-1}=\nu_j$, and the definition of $\cdot$ yield
\begin{gather*}
(\alpha_j,\beta_j,B_j,A_j)=(\tau_j\alpha_{j-1},\:\tau_j\beta_{j-1},\:B_{j-1}\tau_j,\:A_{j-1}\tau_j)\in\Psi_j{,} \\
A=A\tau_j\subseteq A_{j-1}\tau_j=A_j{,} \qquad\text{and}\qquad B=B\tau_j\subseteq B_{j-1}\tau_j=B_j{.}
\end{gather*}
This shows \eqref{eq:h-cross-1} for $j$.

Now, for each $j\in[k]$, property \eqref{eq:h-cross-1} and the definition of $\lambda(u_{j-1},u_j)$ yield an incomparable pair $x_j,y_j\in X(u_{j-1})\setminus X(u_j)$ such that
\[\bagbl{u_{j-1}}{x_j}=\alpha_j{,} \quad \bagbl{u_{j-1}}{y_j}=\beta_j{,} \quad A\subseteq A_j=\labgr{y_j}{u_j}{,} \quad\text{and}\quad B\subseteq B_j=\lablr{x_j}{u_j}{.}\]
This and property \eqref{eq:h-cross-2} imply that the incomparable pairs satisfy the assertion of the lemma.
\end{proof}

We have now gathered all the tools to finish the proof of \Cref{thm:dim-bounded}.
Let $s$ be a decent split of order $p=q^q\cdot 2^{(q+q^2)4^q}$ claimed by \Cref{h-cross}.
Let $d$ be the constant claimed in \Cref{dimension} applied in the context of $s$ for $\ell=k$.
Suppose that $\dim(X,\leq)>d$.
By \Cref{dimension}, for some $h\in[p]$, there is an $h$-chain $u_0\downto\cdots\downto u_k$ of length $k$ and an $h$-cross.
Let $A$ and $B$ be $E_h$-classes provided by \Cref{h-cross}.
For each $j\in[k]$, let $\tau_j=\bagbb{u_{j-1}}{u_j}$, and let $\alpha_j\in A$ and $\beta_j\in B$ be such that $\tau_jA=\set{\alpha_j}$ and $\tau_jB=\set{\beta_j}$.
By \Cref{h-cross}, for each $j\in[k]$, there is an incomparable pair $x_j,y_j\in X(u_{j-1})\setminus X(u_j)$ such that
\[\bagbl{u_{j-1}}{x_j}=\alpha_j{,} \quad \bagbl{u_{j-1}}{y_j}=\beta_j{,} \quad A\subseteq\labgr{y_j}{u_j}{,} \quad\text{and}\quad B\subseteq\lablr{x_j}{u_j}{.}\]
By \Cref{fact:E_h}, the first two equalities above imply $\bagbl{u_i}{x_j}\in A$ and $\bagbl{u_i}{y_j}\in B$ whenever $i<j$.
Therefore, for any $i,j\in[k]$ with $i<j$, we have $\bagbl{u_i}{x_j}\in\labgr{y_i}{u_i}$, which implies $x_j\leq y_i$ by \Cref{fact:L-sets}, and we have $\bagbl{u_i}{y_j}\in\lablr{x_i}{u_i}$, which similarly implies $x_i\leq y_j$.
Moreover, since $(x_1,y_1),\ldots,(x_k,y_k)$ are incomparable pairs, there are no other comparabilities among $x_1,y_1,\ldots,x_k,y_k$.
(Here we use the assumption that $k\geq 3$; if $k=2$, then it could happen that $x_1=y_2$ or $x_2=y_1$.)
We conclude that these elements induce a standard example of dimension $k$ as a subposet of $(X,\leq)$.

\section{Treewidth, dimension, and Kelly examples}\label{sec:kelly}

In this section, we prove the following result, which directly implies \Cref{thm:Kelly-dim-bounded-main}.

\begin{theorem}\label{thm:Kelly-dim-bounded}
For every pair of integers\/ $t\geq 1$ and\/ $k\geq 3$, there is a constant\/ $d\in\setN$ such every poset with cover graph of treewidth\/ $t-1$ and with dimension greater than\/ $d$ contains the Kelly example\/ $\kelly_k$ as a subposet.
\end{theorem}

The following simple fact provides a ``minimal'' set of (in)comparabilities to check in order to exhibit a Kelly example (cf.\ \Cref{fig:standard-and-Kelly}).

\begin{fact}\label{fact:Kelly}
If\/ $k\geq 3$ and elements\/ $a_1,\ldots,a_k,b_1,\ldots,b_k,c_1,\ldots,c_{k-1},d_1,\ldots,d_{k-1}$ of a poset are such that
\begin{enumerate}
\item\label{item:Kelly-1} $a_j\leq c_j\leq b_{j+1}$ and\/ $b_j\geq d_j\geq a_{j+1}$ for every\/ $j\in[k-1]$, and
\item\label{item:Kelly-2} $c_1\leq\cdots\leq c_{k-1}$ and\/ $d_1\geq\cdots\geq d_{k-1}$,
\item\label{item:Kelly-3} $a_j\nleq b_j$ for every\/ $j\in[k]$,
\end{enumerate}
then\/ $a_2,\ldots,a_{k-1},b_2,\ldots,b_{k-1},c_1,\ldots,c_{k-1},d_1,\ldots,d_{k-1}$ induce\/ $\kelly_k$ as a subposet.
\end{fact}

\begin{proof}
All comparabilities required by the definition of $\kelly_k$ follow from \ref{item:Kelly-1} and \ref{item:Kelly-2} by transitivity.
It is straightforward to verify that any further comparability among the elements $a_2,\ldots,a_{k-1},b_2,\ldots,b_{k-1},c_1,\ldots,c_{k-1},d_1,\ldots,d_{k-1}$ would imply $a_j\leq b_j$ for some $j\in[k]$, thus contradicting \ref{item:Kelly-3}.
For instance, we have the following implications:
\begin{itemize}
\item for $1\leq j\leq k-1$, if $c_j=c_{j+1}$, then $a_{j+1}\leq c_{j+1}=c_j\leq b_{j+1}$;
\item for $1\leq i<j\leq k-1$, if $a_j\leq c_i$, then $a_j\leq c_i\leq c_{j-1}\leq b_j$;
\item for $2\leq i<j\leq k-1$, if $a_i\leq a_j$ (in particular, when $c_i\leq a_j$), then $a_i\leq a_j\leq d_{j-1}\leq d_i\leq b_i$.
\end{itemize}
The other cases are analogous.
\end{proof}

We need the following simple regularization of tree decompositions.

\begin{fact}\label{fact:adjust}
Every graph\/ $(X,E)$ with treewidth\/ $t-1$ has a tree decomposition\/ $(T,B)$ of width\/ $t-1$ such that\/ $T$ is a binary tree, the leaf set of\/ $T$ is\/ $X$, and\/ $x\in B(x)$ for every\/ $x\in X$.
\end{fact}

\begin{proof}
Let $(T',B')$ be an arbitrary tree decomposition of $(X,E)$ of width $t-1$.
For every $x\in X$, let $\phi(x)$ be the highest (i.e., minimal in the order $\downtoeq[T']$) node of $T'$ such that $x\in B'(\phi(x))$.
Let $T''$ be the rooted tree on the image $\phi(X)$ such that the tree order $\downtoeq[T'']$ is induced from the tree order $\downtoeq[T']$.
For every node $v$ of $T'$, the closest ancestor $u$ of $v$ in $T'$ that is also a node of $T''$ satisfies $B'(v)\subseteq B'(u)$.
Therefore, $(T'',B'|_{\phi(X)})$ is also a tree decomposition of $(X,E)$ of width $t-1$.

Now, to obtain a binary tree $T$, perform the following replacement for every node $u$ of $T''$.
Let $X_u=\set{x\in X\colon\phi(x)=u}$, and let $\ell$ be the number of children of $u$ in $T''$.
Replace $u$ by an arbitrary binary tree on $\ell-1+\size{X_u}$ nodes, and reconnect the original children of $u$ to that tree so that $\ell-1$ of the new nodes get exactly two children and $\size{X_u}$ of the new nodes get no children (i.e., they are becoming leaves).
Identify the $\size{X_u}$ new leaves arbitrarily with the elements of $X_u$.
Assign a bag $B(v)=B'(u)$ to every node $v$ obtained from $u$ by this replacement; in particular $B(x)=B'(\phi(x))\ni x$ for every $x\in X$.
This gives rise to a tree decomposition $(T,B)$ of $(X,E)$ of width $t-1$ which satisfies the desired conditions.
\end{proof}

The remainder of this section is devoted to the proof of \Cref{thm:Kelly-dim-bounded}.
We fix integers $t$ and $k$, a poset $(X,\leq)$, and a tree decomposition $(T,B)$ of the cover graph of $(X,\leq)$ of width $t-1$ provided by \Cref{fact:adjust}.
In particular, $T$ is a binary tree, the leaf set of $T$ is $X$, and $x\in B(x)$ for every $x\in X$.
We also assume that $\size{X}\geq 2$.

First, we construct an NLC-decomposition of $(X,\leq)$ of width $4^t$.
For every node $u$ of $T$, we enumerate $B(u)$ arbitrarily as
\[B(u)=\set{z^u_1,\ldots,z^u_t}{;}\]
in case $\size{B(u)}<t$, we just repeat some elements in the enumeration.
We define the label set $Q$ of the aimed NLC-decomposition as follows:
\[Q=2^{[t]}\times 2^{[t]}{.}\]
For each $\pi\in Q$, we let $\pi^\geq$ and $\pi^\leq$ denote the two coordinates of $\pi$, so that $\pi=(\pi^\geq,\pi^\leq)$ with $\pi^\geq,\pi^\leq\subseteq[t]$.
For every $x\in X$, the initial label of $x$ is
\[\initial{x}=\bigl(\set{i\in[t]\colon z^x_i\geq x},\;\set{i\in[t]\colon z^x_i\leq x}\bigr){.}\]
Observe that a pair of relations $\sigma=(\sigma^\geq,\sigma^\leq)$ with $\sigma^\geq,\sigma^\leq\subseteq[t]\times[t]$ can be understood as a reverse function $\sigma\colon Q\from Q$ the application of which on a label $\pi=(\pi^\geq,\pi^\leq)\in Q$ is the coordinatewise image: $\sigma\pi=(\sigma^\geq\pi^\geq,\sigma^\leq\pi^\leq)$.
With this understanding, for nodes $u\downto v$ of $T$, we define the relabeling $\bagbb{u}{v}=(\baggb{u}{v},\baglb{u}{v})$ as follows:
\[\baggb{u}{v}=\set{(i,j)\in[t]\times[t]\colon z^u_i\geq z^v_j}{,} \qquad \baglb{u}{v}=\set{(i,j)\in[t]\times[t]\colon z^u_i\leq z^v_j}{.}\]
It follows from \Cref{fact:convexity} that $u\downto v\downto w$ entails $\bagbb{u}{w}=\bagbb{u}{v}\bagbb{v}{w}$.
Furthermore, for every node $u$ of $T$ and every $x\in X(u)$, the label of $x$ at $u$ is $\bagbl{u}{x}=(\baggl{u}{x},\bagll{u}{x})$, where
\[\baggl{u}{x}=\set{i\in[t]\colon z^u_i\geq x}{,} \qquad \bagll{u}{x}=\set{i\in[t]\colon z^u_i\leq x}{.}\]
Finally, for every node $u$ of $T$, we define the status relations
\[R(u)=\set{(\alpha,\beta)\in Q\times Q\colon\alpha^\geq\cap\beta^\leq\neq\emptyset}{,} \qquad R'(u)=\set{(\alpha,\beta)\in Q\times Q\colon\alpha^\leq\cap\beta^\geq\neq\emptyset}{.}\]
That these status relations satisfy the conditions stated in the definition of NLC-decomposition follows again from \Cref{fact:convexity}.
We have thus obtained an NLC-decomposition $(T,Q,\eta,\rho,R,R')$ of $(X,\leq)$ of width $\size{Q}=4^t$.

While in the previous section, we were considering posets with a ``generic'' NLC-decomposition of bounded width, here we will apply the notation and results of the previous section (in particular, \Cref{dimension} and \Cref{h-cross}) to the specific NLC-decomposition $(T,Q,\eta,\rho,R,R')$ constructed above.
In this context, the sets $\lablr{x}{v}$ and $\labgr{x}{v}$ defined in the previous section have the following form.

\begin{fact}\label{fact:L-sets-tw}
For every node\/ $v$ of\/ $T$ and for every\/ $x\in X\setminus X(v)$, we have
\begin{align*}
\lablr{x}{v} &= \set{\pi=(\pi^\geq,\pi^\leq)\in Q\colon\text{there is }j\in\pi^\leq\text{ such that }x\leq z^v_j}\text{,} \\
\labgr{x}{v} &= \set{\pi=(\pi^\geq,\pi^\leq)\in Q\colon\text{there is }j\in\pi^\geq\text{ such that }x\geq z^v_j}\text{.}
\end{align*}
\end{fact}

\begin{proof}
Let $u$ be the lowest common ancestor of $x$ and $v$ in $T$.
Suppose that $x$ is a left descendant and $v$ is a right descendant of $u$.
By definition, we have
\begin{align*}
\lablr{x}{v} &= \set{\pi=(\pi^\geq,\pi^\leq)\in Q\colon(\bagbl{u}{x},\bagbb{u}{v}\pi)\in R(u)} \\
&= \set{\pi=(\pi^\geq,\pi^\leq)\in Q\colon\baggl{u}{x}\cap\baglb{u}{v}\pi^\leq\neq\emptyset} \\
&= \set{\pi=(\pi^\geq,\pi^\leq)\in Q\colon\text{there are }i,j\in[t]\text{ such that }x\leq z^u_i\leq z^v_j\text{ and }j\in\pi^\leq} \\
&= \set{\pi=(\pi^\geq,\pi^\leq)\in Q\colon\text{there is }j\in\pi^\leq\text{ such that }x\leq z^v_j}{,}
\end{align*}
where the last equality follows from \Cref{fact:convexity}.
The other cases are analogous.
\end{proof}

Now, we proceed with the proof of \Cref{thm:Kelly-dim-bounded}.
Let $d$ be the constant claimed in \Cref{dimension} applied in the context of $s$ for $\ell=k$.
Suppose that $\dim(X,\leq)>d$.
By \Cref{dimension}, there is an $h$-chain $u_0\downto\cdots\downto u_k$ of length $k$ and an $h$-cross.
Let $A$ and $B$ be $E_h$-classes provided by \Cref{h-cross}.
For each $j\in[k]$, let $\tau_j=\bagbb{u_{j-1}}{u_j}$, and let $\alpha_j\in A$ and $\beta_j\in B$ be such that $\tau_jA=\set{\alpha_j}$ and $\tau_jB=\set{\beta_j}$.
By \Cref{h-cross}, there are incomparable pairs $x_j,y_j\in X(u_{j-1})\setminus X(u_j)$ for $j\in[k]$ such that the following conditions hold:
\begin{enumerate}
\item for every $j\in[k]$, we have
\[\bagbl{u_{j-1}}{x_j}=\alpha_j{,} \quad \bagbl{u_{j-1}}{y_j}=\beta_j{,} \quad A\subseteq\labgr{y_j}{u_j}{,} \quad\text{and}\quad B\subseteq\lablr{x_j}{u_j}{;}\]
\item for every $j\in[k]\setminus{1}$, we have
\[\labgr{y_j}{u_j}=\labgr{y_{j-1}}{u_{j-1}}\tau_j \quad\text{and}\quad \lablr{x_j}{u_j}=\lablr{x_{j-1}}{u_{j-1}}\tau_j{.}\]
\end{enumerate}

We aim at using \Cref{fact:Kelly} with $x_1,y_1,\ldots,x_k,y_k$ playing the role of $a_1,b_1,\ldots,a_k,b_k$ to exhibit $\kelly_k$ as a subposet of $(X,\leq)$.
To this end, we are going to find indices $r_j\in\beta_{j+1}^\leq$ and $s_j\in\alpha_{j+1}^\geq$ for $j\in[k-1]$ with the following properties, corresponding to conditions \ref{item:Kelly-1} and \ref{item:Kelly-2} in \Cref{fact:Kelly}:
\begin{enumerate}
\item $x_j\leq z^{u_j}_{r_j}\leq y_{j+1}$ and $y_j\geq z^{u_j}_{s_j}\geq x_{j+1}$ for every $j\in[k-1]$, and
\item $z^{u_1}_{r_1}\leq\cdots\leq z^{u_{k-1}}_{r_{k-1}}$ and $z^{u_1}_{s_1}\geq\cdots\geq z^{u_{k-1}}_{s_{k-1}}$.
\end{enumerate}
Here, the definitions of $\beta_{j+1}^\leq$ and $\alpha_{j+1}^\geq$ already imply $z^{u_j}_{r_j}\leq y_{j+1}$ and $z^{u_j}_{s_j}\geq x_{j+1}$ for $j\in[k-1]$ for all choices of $r_j\in\beta_{j+1}^\leq$ and $s_j\in\alpha_{j+1}^\geq$.

We find suitable indices $r_j$ and $s_j$ by induction on $j$.
Since $\beta_2\in B\subseteq\lablr{x_1}{u_1}$, it follows from \Cref{fact:L-sets-tw} that there is $r_1\in\beta_2^\leq=\bagll{u_1}{y_2}$ such that $x_1\leq z^{u_1}_{r_1}$.
Analogously, there is $s_1\in\alpha_2^\geq$ such that $y_1\geq z^{u_1}_{s_1}$.

Now, let $j\in\set{2,\ldots,k-1}$, and suppose we have found suitable indices $r_{j-1}\in\beta_j^\leq$ and $s_{j-1}\in\alpha_j^\geq$.
Since $\beta_j^\leq=\tau_j^\leq\beta_{j+1}^\leq$, there is $r_j\in\beta_{j+1}^\leq$ such that $(r_{j-1},r_j)\in\tau_j^\leq=\baglb{u_{j-1}}{u_j}$, which means (by definition) that $z^{u_{j-1}}_{r_{j-1}}\leq z^{u_j}_{r_j}$.
This also entails $r_{j-1}\in\tau_j^\leq\set{r_j}$.
Consequently, by \Cref{fact:L-sets-tw} and the assumption that $x_{j-1}\leq z^{u_{j-1}}_{r_{j-1}}$, we have $\tau_j(\emptyset,\set{r_j})=(\emptyset,\tau_j^\leq\set{r_j})\in\lablr{x_{j-1}}{u_{j-1}}$.
This and the fact that $\lablr{x_{j-1}}{u_{j-1}}\tau_j=\lablr{x_j}{u_j}$ yield $(\emptyset,\set{r_j})\in\lablr{x_j}{u_j}$, which entails $x_j\leq z^{u_j}_{r_j}$ by \Cref{fact:L-sets-tw}.
Analogously, there is $s_j\in\alpha_{j+1}^\geq$ such that $z^{u_{j-1}}_{s_{j-1}}\geq z^{u_j}_{s_j}$ and $y_j\geq z^{u_j}_{s_j}$.
We have found requested indices $r_j$ and $s_j$, completing the induction step and thus the construction of $\kelly_k$ as a subposet of $(X,\leq)$.

\section{Excluding a minor}\label{sec:minors}

In this section, we prove \Cref{thm:excluding-a-minor}.

We use the following convenient notation for paths.
Let $F$ be a forest and let $u$ and $v$ be two vertices in the same component of $F$.
We let $F[u,v]$ denote the unique path between $u$ and $v$ in $F$, including both $u$ and $v$, and we let $F(u,v)=F[u,v]-\set{u,v}$.
We also use the following mixed versions of this notation: $F[u,v)=F[u,v]-\set{v}$ and $F(u,v]=F[u,v]-\set{u}$.

\begin{lemma}
\label{lem:Kelly_poset_to_minor}
For every integer\/ $k\geq 3$, there is an integer\/ $N=N(k)\geq 3$ such that for every poset\/ $P$ containing\/ $\kelly_N$ as a subposet, the cover graph of\/ $P$ contains the cover graph of\/ $\kelly_k$ as a minor.
\end{lemma}

\begin{proof}
Fix a positive integer $k\geq 3$.
Let $P$ be a poset containing $\kelly_N$ as a subposet, where $N$ is a large enough constant to be specified later.
We denote the elements of $P$ inducing $\kelly_N$ by $a_2,\ldots,a_{N-1},b_2,\ldots,b_{N-1},c_1,\ldots,c_{N-1},d_1,\ldots,d_{N-1}$, in such a way that they satisfy the inequalities from the definition in~\Cref{sec:intro}.

Let $G$ be the cover graph of $P$.
It will be convenient to adapt the notations for paths introduced before the lemma to the setting of cover graphs as follows: for any two elements $u$ and $v$ with $u<v$ in $P$, fix arbitrarily one covering chain $u=z_1<\cdots<z_m=v$ in $P$ for some $m\geq 2$.
This chain corresponds to a path in $G$, which we denote by $G[u,v]$.
(The variants $G(u,v)$, $G[u,v)$, and $G(u,v]$ are defined naturally.)
Note that this path is naturally ordered from $u$ to $v$, which will be used implicitly in the proof.

The Grid Minor Theorem of Robertson and Seymour~\cite{RS86} implies that graphs excluding a fixed planar graph $H$ as a minor have treewidth bounded by some constant depending on $\size{V(H)}$.
Thus, since the cover graph of $\kelly_k$ is planar, there is a constant $t$ depending only on $k$ such that every graph with treewidth at least $t$ contains the cover graph of $\kelly_k$ as a minor.

Let $I=\set{2,3,\ldots,N-1}$.
It can be verified that each of the sets
\[\set{G[a_i,c_i]\colon i\in I}{,} \qquad \set{G[c_{i-1},b_i]\colon i\in I}{,} \qquad \set{G[a_i,d_{i-1}]\colon i\in I}{,} \qquad \set{G[d_i,b_i]\colon i\in I}\]
consists of pairwise vertex-disjoint paths in $G$, as otherwise it would create a non-existent relation in $\kelly_N$.
For each $i\in I$, let
\begin{align*}
Z_{i,1} &= G[a_i,c_i]{,} &
Z_{i,2} &= G[a_i,d_{i-1}]{,} &
Z_{i,3} &= G[c_{i-1},b_i]{,} &
Z_{i,4} &= G[d_i,b_i]{.}\\
\intertext{Also let}
H_1 &= \bigcup_{i\in I} Z_{i,1}{,} &
H_2 &= \bigcup_{i\in I} Z_{i,2}{,} &
H_3 &= \bigcup_{i\in I} Z_{i,3}{,} &
H_4 &= \bigcup_{i\in I} Z_{i,4}{.}
\end{align*}
Note that $H_\alpha$ is a forest (collection of paths) for each $\alpha\in[4]$.
For each $i\in I$, let
\[Z_i = \bigcup_{\alpha\in[4]}Z_{i,\alpha}{.}\]

Consider the complete graph $K_{\size{I}}$ having $I$ as the vertex set.
Let $\phi$ be the coloring of the edges of $K_{\size{I}}$ defined as follows.
For any $i,j\in I$ with $i<j$,
\[\phi(\set{i,j}) = \begin{cases}
0&\text{if }Z_i\text{ and }Z_j\text{ are vertex disjoint; and}\\
(\alpha,\beta)&\parbox[t]{273pt}{otherwise, where $\alpha,\beta\in[4]$ are chosen in such a way that \\
$V(Z_{i,\alpha})\cap V(Z_{j,\beta})\neq\emptyset$.}
\end{cases}\]
Note that $\phi$ uses at most $17$ colors.
By Ramsey's Theorem~\cite{Ramsey30}, if $N$ is chosen large enough as a function of $t$, there is a set $J\subseteq I$ such that $\phi$ uses the same color for all $\set{i,j}$ with
$i,j\in J$ and $i<j$, and such that $J$ has size $k$ if that color is $0$, and size $2t+2$ otherwise.

Fix such a monochromatic set $J\subseteq I$.
First, suppose that the color on all the edges in $J$ is $0$.
In this case, we aim to build a model of the cover graph of $\kelly_k$ in $G$.
Let $J=\set{j_1,\ldots,j_k}$ with $j_1<\cdots<j_k$.

Recall that $c_1<\cdots<c_{N-1}$ and $d_1\cdots<d_{N-1}$ in $\kelly_N$.
Let $C$ be the path in $G$ obtained from the concatenation of the paths $G[c_1,c_2],\ldots,G[c_{N-2},c_{N-1}]$.
(Observe the concatenation is indeed a path, since $c_1<\cdots<c_{N-1}$ in $P$.)
Let $D$ be the path in $G$ obtained from the concatenation of the paths $G[d_1,d_2],\ldots,G[d_{N-2},d_{N-1}]$.
Note that paths $C$ and $D$ are vertex-disjoint, as any intersection between them would imply a non-existent relation in $\kelly_N$.

For each $i\in I$, let $a'_i$ be the first vertex of $G[a_i,c_i]$ that lies on $C$, and let $b'_i$ be the last vertex of $G[c_{i-1},b_i]$ that lies on $C$.
We claim that
\begin{equation}
\label{eq:c_i-1__c_i}
c_{i-1}\leq b'_i<a'_i\leq c_i\text{ in }P \quad \text{for all }i\in I{.}
\end{equation}
Indeed, the first inequality follows from the fact that
$c_{i-1}$ lies in $G[c_{i-1},b_i]$.
The second inequality follows as otherwise we would have $a'_i\leq b'_i$ in $P$ and hence $a_i\leq a'_i\leq b'_i\leq b_i$ in $P$, a contradiction.
The third inequality follows from the fact that $c_i$ lies in $G[a_i,c_i]$.
Similarly, we define for each $i\in I$, $a''_i$ to be the first vertex of $G[a_i,d_{i-1}]$ that lies on $D$, and $b''_i$ to be the last vertex of $G[d_i,b_i]$ that lies on $D$.
We claim that
\begin{equation}
\label{eq:d_i_plus_1__d_i}
d_i\leq b''_i<a''_i\leq d_{i-1}\text{ in }P \quad \text{for all }i\in I{.}
\end{equation}
The proof is symmetric to the one above.
Finally, we remark that $a_i\neq a'_i$ for each $i\in I$, and thus $a_i<a'_i$ in $P$.
This is because, if $a_i=a'_i$, then we would have $a_{i-1}<c_{i-1}\leq a'_i=a_i$ in $P$, a contradiction.
By a symmetric argument, we have $a_i\neq a''_i$, $b_i\neq b'_i$, and $b_i\neq b''_i$ for each $i\in I$.

Now, we build a model $\cgM$ of the cover graph of $\kelly_k$ in $G$.
The model $\cgM$ is a collection of connected subgraphs $M_x$ of $G$, one for each element $x$ of $\kelly_k$, defined as follows:
\[M_x = \begin{cases}
H_1[a_{j_i},a'_{j_i})\cup H_2[a_{j_i},a''_{j_i}) & \text{if }x=a_i\text{ and }i\in\set{2,\ldots,k-1}{,} \\
H_3[b_{j_i},b'_{j_i})\cup H_4[b_{j_i},b''_{j_i}) & \text{if }x=b_i\text{ and }i\in\set{2,\ldots,k-1}{,} \\
C[a'_{j_i},a'_{j_{i+1}}) & \text{if }x=c_i\text{ and }i\in\set{1,\ldots,k-1}{,} \\
D[a''_{j_{i+1}},a''_{j_i}) & \text{if }x=d_i\text{ and }i\in\set{1,\ldots,k-1}{.}
\end{cases}\]

It remains to show that $\cgM$ is indeed a model of the cover graph of $\kelly_k$ in $G$.
To prove this, we need to show that the subgraphs $M_x$ defined above are non-empty and connected, that they are pairwise vertex-disjoint, and finally, that the necessary edges between them exist in $G$ to obtain the desired model.
The first property holds, since each $M_x$ is either a non-empty path, or the union of two non-empty paths sharing an endpoint.
(The paths are non-empty since $a_i\neq a'_i$, $a_i\neq a''_i$, $b_i\neq b'_i$, and $b_i\neq b''_i$ for each $i\in I$, and since $a'_2<a'_3<\cdots<a'_{N-1}$ in $P$ and $a''_{N-1}<a''_{N-2}<\cdots<a''_2$ in $P$.)
That the subgraphs $M_x$ defined in the first two cases above are pairwise vertex-disjoint follows from the fact that the edges in $J$ are colored $0$.
That the subgraphs $M_x$ defined in the last two cases are vertex-disjoint follows from the vertex-disjointness of paths $C$ and $D$.
Moreover, given the definitions of $a'_i$, $a''_i$, $b'_i$, and $b''_i$ for $i\in I$, it follows that each subgraph $M_x$ defined in one of the first two cases above avoids both $C$ and $D$, and thus $M_x$ is vertex-disjoint from every subgraph defined in the last two cases.
Finally, let us check that the necessary edges connecting distinct subgraphs $M_x$ exist in $G$, namely:
\begin{itemize}
\item an edge between $H_1[a_{j_i},a'_{j_i})\cup H_2[a_{j_i},a''_{j_i})$ and $C[a'_{j_i},a'_{j_{i+1}})$ for $i\in\set{2,\ldots,k-1}$;
\item an edge between $H_1[a_{j_i},a'_{j_i})\cup H_2[a_{j_i},a''_{j_i})$ and $D[a''_{j_i},a''_{j_{i-1}})$ for $i\in\set{2,\ldots,k-1}$;
\item an edge between $H_3[b_{j_i},b'_{j_i})\cup H_4[b_{j_i},b''_{j_i})$ and $C[a'_{j_{i-1}},a'_{j_i})$ for $i\in\set{2,\ldots,k-1}$;
\item an edge between $H_3[b_{j_i},b'_{j_i})\cup H_4[b_{j_i},b''_{j_i})$ and $D[a''_{j_{i+1}},a''_{j_i})$ for $i\in\set{2,\ldots,k-1}$;
\item an edge between $C[a'_{j_i},a'_{j_{i+1}})$ and $C[a'_{j_{i+1}},a'_{j_{i+2}})$ for $i\in\set{1,\ldots,k-2}$;
\item an edge between $D[a''_{j_{i+2}},a''_{j_{i+1}})$ and $D[a''_{j_{i+1}},a''_{j_i})$ for $i\in\set{1,\ldots,k-2}$.
\end{itemize}
In the first two cases and last two cases, the required edge clearly exists in $G$ given the definitions of the two subgraphs under consideration.
In the remaining two cases, the desired edge exists as well: this follows from the fact that, for each $i\in I$, the path $C[a'_{i-1},a'_i)$ contains $b'_i$ (as follows from \eqref{eq:c_i-1__c_i}) and the path $D[a''_i,a''_{i-1})$ contains $b''_{i-1}$ (as follows from \eqref{eq:d_i_plus_1__d_i}).
Therefore, $\cgM$ is a model of the cover graph of $\kelly_k$ in $G$, as desired.

Next, assume that the edges of $J$ are not colored $0$, and let $(\alpha,\beta)$ be their color.
Recall that $\size{J}=2t+2$ in this case.
Let $U$ be the first $t+1$ integers in $J$, and let $V=J-U$.
Thus, $V(Z_{i,\alpha})\cap V(Z_{j,\beta})\neq\emptyset$ for every $i\in U$ and $j\in V$.
Hence, there exist a set $S_\alpha$ of $t+1$ components (which are paths) of $H_\alpha$ and a set $S_\beta$ of $t+1$ components of $H_\beta$ (also paths) such that each path in $S_\alpha$ intersects each path in $S_\beta$.
Let $\cgB$ denote the collection of the subgraphs of $G$ consisting of the union of one path from $S_\alpha$ and one from $S_\beta$.
Observe that every subgraph in $\cgB$ is connected, and any two subgraphs in $\cgB$ have at least one vertex in common.
Thus, in the terminology of graph minors, $\cgB$ is a \emph{bramble}.
If $X$ is a vertex subset of $G$ of size at most $t$, then $X$ misses some path in $S_\alpha$ and some path in $S_\beta$, and thus $X$ misses their union, which is a member of~$\cgB$.
We deduce that every vertex subset of $G$ hitting all members of $\cgB$ has size at least $t+1$.
By a standard lemma on brambles (see e.g.~\cite{ST93}), this shows that $G$ has treewidth at least $t$.
Therefore, by the Grid Minor Theorem, $G$ contains the cover graph of $\kelly_k$ as a minor.
\end{proof}

We can now prove \Cref{thm:excluding-a-minor}.

\begin{proof}[Proof of \Cref{thm:excluding-a-minor}]
Let $\cgC$ be a minor-closed class of graphs.
If $\cgC$ contains the cover graphs of all Kelly examples, then posets with cover graphs in $\cgC$ have unbounded dimension.
Now, suppose that $\cgC$ excludes the cover graph of $\kelly_k$ for some $k\geq 3$.
It remains to show that posets with cover graphs in $\cgC$ have dimension bounded by a function of $k$.
Let $P$ be such a poset, and let $G$ be its cover graph.
Since the cover graph of $\kelly_k$ is planar and $G$ excludes it as a minor, by the Grid Minor Theorem, we may assume that $G$ has treewidth less than $t$ for some constant $t=t(k)$.
By \Cref{lem:Kelly_poset_to_minor}, $P$ does not contain $\kelly_N$ as a subposet, where $N=N(k)$ is the constant from that lemma.
Now, by \Cref{thm:Kelly-dim-bounded-main}, $P$ has dimension at most $f(t, N)$, where $f$ is the function from that theorem.
\end{proof}

\section{Boolean dimension}\label{sec:bdim}

In this section, we prove our results about Boolean dimension of posets of bounded NLC-width and with cover graphs of bounded treewidth. Before we proceed to these results, we first introduce some auxiliary tools about encoding semigroup-labeled trees through linear orders. These tools already contain the combinatorial core of the reasoning: Colcombet's Theorem.

\subsection*{Encoding and decoding trees}

We now present two lemmas about implementing various queries in trees using linear orders. We will use the following common definition. For a tree $S$ with leaf set $Z$, a linear order $\preceq$ on $Z$ is \emph{consistent} with $S$ if for every node $u$ of $S$, the set $Z(u)=\{z\in Z\colon z$ is a descendant of $u$ in $S\}$ is a contiguous interval of $\preceq$. Also, recall that for an assertion $\psi$, we write $[\psi]$ for the Boolean value of $\psi$: $[\psi]=1$ if $\psi$ holds, and $[\psi]=0$ otherwise.

The first lemma is about implementing basic navigation in a tree. In essence, the argument is the same as in~\cite[Color Detection]{FMM20}.

\begin{lemma}\label{bdim-aux-tree}
Let\/ $C$ be a finite set of colors and\/ $k=2\lceil\log_2\size{C}\rceil$.
Then there exists a function\/ $\dir\colon\set{0,1}^{k+1}\to C^2$ such that the following holds.
For every tree\/ $S$ with leaf set\/ $Z$, every linear order\/ $\refpreceq$ on\/ $Z$ consistent with\/ $S$, and every coloring\/ $\phi$ of non-root nodes of\/ $S$ with colors from\/ $C$, there exist linear orders\/ $\preceq^\dir_1,\ldots,\preceq^\dir_k$ such that for all distinct\/ $x,y\in Z$, we have
\[\dir\bigl([x\refpreceq y],\:[x\preceq^\dir_1y],\:\ldots,\:[x\preceq^\dir_ky]\bigr)=(\phi(v),\phi(w))\text{,}\]
where\/ $v$ and\/ $w$ are the children of the lowest common ancestor of\/ $x$ and\/ $y$ such that\/ $v$ is an ancestor of\/ $x$ and\/ $w$ is an ancestor of\/ $y$.
\end{lemma}

\begin{proof}
Let $\ell=\lceil\log_2\size{C}\rceil=k/2$.
Without loss of generality we assume that $C$ consists of distinct binary strings, each of length $\ell$.
Also, for $i\in[\ell]$ and $t\in\set{0,1}$, let $C^t_i$ comprise those colors of $C$ whose $i$th symbol is $t$.

For every $i\in[\ell]$, we construct two linear orders on $Z$: $\preceq^\dir_{2i+1}$ and $\preceq^\dir_{2i+2}$.
To define them, consider any distinct $x,y\in Z$.
Let $u$ be the lowest common ancestor of $x$ and $y$ in $S$, and let $v$ and $w$ be the children of $u$ such that $x$ is a descendant of $v$ and $y$ is a descendant of $w$.
Let $c=\phi(v)$ and $d=\phi(w)$.
We define
\begin{align*}
x&\preceq^\dir_{2i+1}y \quad\Longleftrightarrow\quad (c\in C^0_i\text{ and }d\in C^1_i)\text{ or }(c,d\in C^0_i\text{ and }x\refpreceq y)\text{ or }(c,d\in C^1_i\text{ and }y\refpreceq x){,} \\
x&\preceq^\dir_{2i+2}y \quad\Longleftrightarrow\quad (c\in C^1_i\text{ and }d\in C^0_i)\text{ or }(c,d\in C^0_i\text{ and }x\refpreceq y)\text{ or }(c,d\in C^1_i\text{ and }y\refpreceq x){.}
\end{align*}
Since $\refpreceq$ is consistent with $S$, both $\preceq^\dir_{2i+1}$ and $\preceq^\dir_{2i+2}$ are linear orders on $Z$.

It remains to describe $\dir$.
For this, consider any distinct $x,y\in Z$, and let $c$ and $d$ be defined as in the previous paragraph.
For every $i\in[\ell]$, the triple of Boolean values $[x\refpreceq y]$, $[x\preceq^\dir_{2i+1}y]$, and $[x\preceq^\dir_{2i+2}y]$ suffices to uniquely determine the $i$th symbols of $c$ and $d$, as follows:
\begin{itemize}
\item the $i$th symbol of $c$ is $[x\succeq^\dir_{2i+1}y]$ if $x\refpreceq y$, and it is $[x\preceq^\dir_{2i+2}y]$ otherwise;
\item the $i$th symbol of $d$ is $[x\succeq^\dir_{2i+2}y]$ if $x\refpreceq y$, and it is $[x\preceq^\dir_{2i+1}y]$ otherwise.
\end{itemize}
Thus, from the $(k+1)$-tuple of Boolean values $[x\refpreceq y],[x\preceq^\dir_1y],\ldots,[x\preceq^\dir_ky]$, we can uniquely determine every symbol of $c$ and of $d$, so also $c$ and $d$ themselves.
This mapping constitutes the sought function $\dir$.
\end{proof}

The second lemma is the key step in the proof.
It allows us to implement node-to-leaf composition queries in a semigroup-labeled tree.
The proof relies on Colcombet's Theorem combined with \Cref{bdim-aux-tree}.

\begin{lemma}\label{bdim-aux-composition}
For every finite semigroup\/ $(\Lambda,\cdot)$, there is a function\/ $\agg\colon\set{0,1}^{d+1}\to\Lambda^2$, where
\[d=(6\size{\Lambda}+2)\lceil\log_2\size{\Lambda}\rceil+6\size{\Lambda}\text{,}\]
such that the following holds: for every tree\/ $T$ with leaf set\/ $X$, every linear order\/ $\refpreceq$ on\/ $X$ consistent with\/ $T$, and every\/ $(\Lambda,\cdot)$-labeling\/ $\lambda$ of\/ $T$, there exist linear orders\/ $\preceq^\agg_1,\ldots,\preceq^\agg_d$ on\/ $X$ such that for all distinct\/ $x,y\in X$, we have
\[\agg\bigl([x\refpreceq y],\:[x\preceq^\agg_1y],\:\ldots,\:[x\preceq^\agg_dy]\bigr)=(\lambda(u,x),\lambda(u,y))\text{,}\]
where\/ $u$ is the lowest common ancestor of\/ $x$ and\/ $y$ in\/ $T$.
\end{lemma}

\begin{proof}
Let $V$ be the set of nodes of $T$ and $U$ be the set of inner nodes of $T$.
For a set $W$ with $V\setminus U\subseteq W\subseteq V$, let $T[W]$ denote the rooted subtree induced from $T$ on $W$ in the tree-order sense, i.e., the rooted tree obtained from $T$ by taking $W$ as the set of nodes and the order ${\downtoeq}|_{W\times W}$ as $\downtoeq[{T[W]}]$.
Note that $X$ is the set of leaves in $T[W]$ and the linear order $\refpreceq$ on $X$ is consistent with $T[W]$.

Let $p=\size{\Lambda}$ and $k=2\lceil\log_2\size{\Lambda}\rceil$.
By Colcombet's Theorem, there exists a split $s\colon U\to[p]$ that is forward Ramseyan with respect to $\lambda$. We fix $s$ for the remainder of the proof.
For convenience, we extend $s$ on the entire set of nodes $V$ by setting $s(v)=p+1$ for all nodes $v\in V\setminus U$ (i.e., the root and the leaves of $T$).

Given any two distinct leaves $x,y\in X$, let $u$ be the lowest common ancestor of $x$ and $y$ in $T$, and for each $h\in[p+1]$, let $v_h$ be the highest node $v$ in $T$ such that $u\downto v\downtoeq x$ and $s(v)\geq h$.
Thus
\[u\downto v_1\downtoeq\cdots\downtoeq v_{p+1}=x{.}\]
Our goal is to determine the value of $\lambda(u,x)$, which we can represent as the semigroup product of $\lambda(u,v_1)$ and the values of $\lambda(v_h,v_{h+1})$ over those $h\in[p]$ for which $v_h\downto v_{h+1}$.

For any $h\in[p+1]$, let $V_{\geq h}=\set{v\in V\colon s(v)\geq h}$; in particular $V_{\geq 1}=V$.
For any $h\in[p]$ and $i\in\set{0,1}$, let $V_h^i=\{v\in U\colon s(v)=h$ and the number of $h$-neighbors $u$ of $v$ with $u\downto v$ is $i\pmod{2}\}$.
For any $h\in[p]$ and any node $v\in V_{\geq h}$ other than the root of $T$, let
\begin{itemize}
\item $\phi_h(v)=\lambda(u,v)$ where $u$ is the parent of $v$ in $T[V_{\geq h}]$, i.e., the lowest node in $V_{\geq h}$ with $u\downto v$;
\item $\psi_h(v)=i\in\set{0,1}$ if $v\in V_h^i$ and $\psi_h(v)=0$ if $s(v)>h$;
\item $\zeta_h(v)=0$ if $s(v)=h$, and $\zeta_h(v)=1$ if $s(v)>h$.
\end{itemize}

In the context of fixed $x,y\in X$, we define $u,v_0,\ldots,v_{p+1}$ as above, and furthermore, for any $h\in[p]$ and $i\in\set{0,1}$, we define $v_h^i$ as the highest node $v$ in $T$ such that $u\downto v\downtoeq x$ and $v\in V_h^i\cup V_{\geq h+1}$ (which implies $v_h\downtoeq v_h^i\downtoeq v_{h+1}$).
Note that $v_h$ is a child of the lowest common ancestor of $x$ and $y$ in $T[V_{\geq h}]$ for $h\in[p+1]$, and so is $v_h^i$ in $T[V_h^i\cup V_{\geq h+1}]$ for $h\in[p]$ and $i\in\set{0,1}$.

We apply \Cref{bdim-aux-tree} to the tree $T$ with coloring $\phi_1$ to obtain $k$ linear orders $\preceq_{0,1}^\dir,\ldots,\preceq_{0,k}^\dir$ such that for any distinct $x,y\in X$, given the Boolean values $[x\refpreceq y],[x\preceq_{0,1}^\dir y],\ldots,[x\preceq_{0,k}^\dir y]$, we can determine the value $\phi_1(v_1)$, which is equal to $\lambda(u,v_1)$, as $u$ is the parent of $v_1$.

For every $h\in[p]$, we apply \Cref{bdim-aux-tree} to the following trees and their colorings, in every application using $\refpreceq$ as the linear order on $X$:
\begin{itemize}
\item $T[V_{\geq h}]$ with coloring $\psi_h$;
\item $T[V_h^0\cup V_{\geq h+1}]$ with coloring $\zeta_h$ restricted to $V_h^0\cup V_{\geq h+1}$;
\item $T[V_h^1\cup V_{\geq h+1}]$ with coloring $\zeta_h$ restricted to $V_h^1\cup V_{\geq h+1}$;
\item $T[V_h^0\cup V_{\geq h+1}]$ with coloring $\phi_h$ restricted to $V_h^0\cup V_{\geq h+1}$;
\item $T[V_h^1\cup V_{\geq h+1}]$ with coloring $\phi_h$ restricted to $V_h^1\cup V_{\geq h+1}$;
\item $T[V_{\geq h+1}]$ with coloring $\phi_h$ restricted to $V_{\geq h+1}$.
\end{itemize}
The first three applications yield six linear orders on $X$, and the last three yield $3k$ linear orders on $X$.
Altogether, this yields $3k+6$ linear orders $\preceq_{h,1}^\dir,\ldots,\preceq_{h,3k+6}^\dir$ on $X$ such that for any distinct $x,y\in X$, given the Boolean values $[x\refpreceq y],[x\preceq_{h,1}^\dir y],\ldots,[x\preceq_{h,3k+6}^\dir y]$, we can determine the values
\[\psi_h(v_h),\:\zeta_h(v_h^0),\:\zeta_h(v_h^1),\:\phi_h(v_h^0),\:\phi_h(v_h^1),\:\phi_h(v_{h+1}){.}\]
They let us determine whether $v_h\downto v_{h+1}$ and if it is so, the value of $\lambda(v_h,v_{h+1})$, as stated in the following claim.

\begin{claim}\label{cl:toProject}
If $\zeta_h(v_h^0)=\zeta_h(v_h^1)=1$, then $v_h=v_{h+1}$.
Otherwise, $v_h\downto v_{h+1}$ and the following holds.
\begin{enumerate}
\item If exactly one of the values $\zeta_h(v_h^0)$ and $\zeta_h(v_h^1)$ is $1$, then $\lambda(v_h,v_{h+1})=\phi_h(v_{h+1})$.
\item If $\zeta_h(v_h^0)=\zeta_h(v_h^1)=0$ and $\psi_h(v_h)=0$ then $\lambda(v_h,v_{h+1})=\phi_h(v_h^1)\cdot\phi_h(v_{h+1})$.
\item If $\zeta_h(v_h^0)=\zeta_h(v_h^1)=0$ and $\psi_h(v_h)=1$ then $\lambda(v_h,v_{h+1})=\phi_h(v_h^0)\cdot\phi_h(v_{h+1})$.
\end{enumerate}
\end{claim}

\begin{proof}
Let $u$ be the lowest common ancestor of $x$ and $y$ in $T$.
We have $v_h=v_h^0$ or $v_h=v_h^1$.
If $v_h\downto v_{h+1}$, then $s(v_h)=h$, so $\zeta_h(v_h)=0$ and thus $\zeta_h(v_h^0)=0$ or $\zeta_h(v_h^1)=0$.
This shows that if $\zeta_h(v_h^0)=\zeta_h(v_h^1)=1$, then $v_h=v_{h+1}$, as claimed.

Now, suppose $\zeta_h(v_h^0)=0$ or $\zeta_h(v_h^1)=0$.
Then $v_h\downto v_{h+1}$.
We distinguish the following cases:
\begin{enumerate}
\item $\zeta_h(v_h^0)=1$ or $\zeta_h(v_h^1)=1$, which means that one of $v_h^0$, $v_h^1$ is $v_{h+1}$.
Then the other of $v_h^0$, $v_h^1$ (which is $v_h$) is the parent of $v_{h+1}$ in $T[V_{\geq h}]$, so $\lambda(v_h,v_{h+1})=\phi_h(v_{h+1})$, as claimed.
\item $\zeta_h(v_h^0)=\zeta_h(v_h^1)=0$ and $\psi_h(v_h)=0$.
Then $v_h=v_h^0$, $s(v_h^1)=h$, and $v_h^0$ is the parent of $v_h^1$ in $T[V_{\geq h}]$.
Let $v'$ be the parent of $v_{h+1}$ in $T[V_{\geq h}]$.
Thus we have $v_h=v_h^0\downto v_h^1\downtoeq v'\downto v_{h+1}$.
By the fact that $s$ is forward Ramseyan, we have
\[\lambda(v_h,v_{h+1}) = \lambda(v_h^0,v_h^1)\cdot\lambda(v_h^1,v')\cdot\lambda(v',v_{h+1}) = \lambda(v_h^0,v_h^1)\cdot\lambda(v',v_{h+1}) = \phi_h(v_h^1)\cdot\phi_h(v_{h+1}){,}\]
as claimed.
\item $\zeta_h(v_h^0)=\zeta_h(v_h^1)=0$ and $\psi_h(v_h)=1$.
This case is analogous to the previous one, with the roles of $v_h^0$ and $v_h^1$ exchanged.\qedhere
\end{enumerate}
\end{proof}

To conclude, for any distinct $x,y\in X$, the Boolean values $[x\refpreceq y],[x\preceq_{0,1}^\dir y],\ldots,[x\preceq_{0,k}^\dir y]$, and $[x\refpreceq y],[x\preceq_{h,1}^\dir y],\ldots,[x\preceq_{h,3k+6}^\dir y]$ for all $h\in[p]$ let us determine $\lambda(u,v_1)$ and the values of $\lambda(v_h,v_{h+1})$ for all those $h\in[p]$ for which $v_h\downto v_{h+1}$, which altogether give $\lambda(u,x)$ by taking the semigroup product.
Analogously (by exchanging the roles of $x$ and $y$), they let us determine $\lambda(u,y)$.
\end{proof}

\subsection*{Posets of bounded NLC-width.}

We now use \Cref{bdim-aux-composition} to prove the following statement, which immediately implies \Cref{thm:bdim-main}.

\begin{theorem}\label{thm:bdim}
For every\/ $q\in\setN$, there is\/ $b\in 2^{\cramped{\Oh(q^2)}}$ such that every poset of NLC-width\/ $q$ has Boolean dimension at most\/ $b$.
\end{theorem}

\begin{proof}
Let $(X,\leq)$ be a poset and $(T,Q,\eta,\rho,R,R')$ be an NLC-decomposition of it of width $q$.
Let $\Lambda$ be the semigroup of backward functions from $Q$ to $Q$ with composition as the semigroup operation $\cdot$.
Thus $\size{\Lambda}=q^q$, and the relabeling assignment $\rho$ is a $(\Lambda,\cdot)$-labeling of $T$.

Let $\refpreceq$ be the left-to-right order on $X$ consistent with $T$ such that for any distinct $x,y\in X$, we have $x\refpreceq y$ if $x$ is a left descendant and $y$ is a right descendant of the lowest common ancestor of $x$ and $y$ in $T$.
Apply \Cref{bdim-aux-composition} to $T$, $\refpreceq$, and the labeling $\rho$.
This yields a function $\agg\colon\set{0,1}^{d+1}\to\Lambda^2$ and $d$ linear orders $\preceq^\agg_1,\ldots,\preceq^\agg_d$, where $d=(6q^q+2)\lceil q\log_2q\rceil+6q^q$, such that for any distinct $x,y\in X$, we have
\[\agg\bigl([x\refpreceq y],\:[x\preceq^\agg_1y],\:\ldots,\:[x\preceq^\agg_dy]\bigr)=(\bagbb{u}{x},\bagbb{u}{y}){,}\]
where $u$ is the lowest common ancestor of $x$ and $y$ in $T$.

Apply \Cref{bdim-aux-tree} to $T$, $\refpreceq$, and the coloring of the non-root nodes of $T$ defined as follows: for any non-leaf node $u$ of $T$, assign $R(u)$ as the color of the left child and $R'(u)$ as the color of the right child of $u$.
This yields a function $\rel\colon\set{0,1}^{4q+1}\to 2^{Q\times Q}\times 2^{Q\times Q}$ and $4q$ linear orders $\preceq^\rel_1,\ldots,\preceq^\rel_{4q}$ such that for any distinct $x,y\in X$, we have
\[\rel\bigl([x\refpreceq y],\:[x\preceq^\rel_1y],\:\ldots,\:[x\preceq^\rel_{4q}y]\bigr)=(R^x(u),R^y(u)){,}\]
where $u$ is the lowest common ancestor of $x$ and $y$ in $T$ and $(R^x(u),R^y(u))=(R(u),R'(u))$ if $x\refpreceq y$ and $(R^x(u),R^y(u))=(R'(u),R(u))$ otherwise.

Finally, let $T'$ be a tree with leaf set $X$ and with no inner vertices, so that every leaf is a child of the root.
Apply \Cref{bdim-aux-tree} to $T'$, $\refpreceq$, and the coloring $\eta$ of $X$ by the initial labels.
This yields a function $\init\colon\set{0,1}^{k+1}\to Q\times Q$ and $k$ linear orders $\preceq^\init_1,\ldots,\preceq^\init_k$, where $k=2\lceil\log_2q\rceil$, such that for any distinct $x,y\in X$, we have
\[\init\bigl([x\refpreceq y],\:[x\preceq^\init_1y],\:\ldots,\:[x\preceq^\init_ky]\bigr)=(\initial{x},\initial{y}){.}\]

Now, by the definition of NLC-decomposition, we have
\[x\leq y\qquad\text{if and only if}\qquad\bigl(\bagbb{u}{x}\initial{x},\:\bagbb{u}{y}\initial{y}\bigr)\in R^x(u){,}\]
which can be decided with the use of functions $\agg$, $\rel$, and $\init$ given
\[[x\refpreceq y],\:[x\preceq^\agg_1y],\ldots,[x\preceq^\agg_dy],\:[x\preceq^\rel_1y],\ldots,[x\preceq^\rel_{4q}y],\:[x\preceq^\init_1y],\ldots,[x\preceq^\init_ky]{.}\]
This constitutes the sought decoder, so we can set $b=d+4q+k+1$.
\end{proof}

\subsection*{Posets with cover graphs of bounded treewidth}

By combining \Cref{thm:bdim} with \Cref{fact:tw-NLC}, we infer that every poset whose cover graph has treewidth $t-1$ has Boolean dimension bounded by $\cramped{2^{2^{\Oh(t)}}}$.
Asymptotically, this matches the upper bound obtained by Felsner, Mészáros, and Micek~\cite{FMM20}.
We now give a direct argument that relies on combining the construction of an NLC-decomposition in \Cref{sec:kelly} with an application of \Cref{bdim-aux-composition}.
This argument improves the bound to single-exponential in $t^2$.

\begin{theorem}
\label{thm:bdim-treewidth}
For every\/ $t\in\setN$, there is\/ $b\in\Oh(2^{\cramped{t^2}}t)$ such that every poset with cover graph of treewidth\/ $t-1$ has Boolean dimension at most\/ $b$.
\end{theorem}

\begin{proof}
Let $(X,\leq)$ be a poset, and let $(T,B)$ be a tree decomposition of the cover graph of $(X,\leq)$ of width $t-1$ provided by \Cref{fact:adjust}.
In particular, the leaf set of $T$ coincides with $X$ and $x\in B(x)$ for every $x\in X$.
For every node $u$ of $T$, enumerate the elements of $B(u)$ arbitrarily as $B(u)=\set{z^u_1,\ldots,z^u_t}$, repeating the vertices in case $\size{B(u)}<t$.
For each $x\in X$, since $x\in B(x)$, we can assume that $x=z^x_1$.

Let $\Lambda=2^{[t]\times[t]}$.
Consider the semigroup $(\Lambda,\cdot)$ of relations on $[t]$ with composition of relations as the semigroup operation.
For any two nodes $u\downto v$, define
\[\baggb{u}{v}=\set{(i,j)\in[t]\times[t]\colon z^u_i\geq z^v_j}{,} \qquad \baglb{u}{v}=\set{(i,j)\in[t]\times[t]\colon z^u_i\leq z^v_j}{.}\]
It follows from \Cref{fact:convexity} that $\rho^\geq$ and $\rho^\leq$ are $(\Lambda,\cdot)$-labelings of $T$.

Next, let $\refpreceq$ be the left-to-right order on $X$ consistent with $T$, i.e., for distinct $x,y\in X$, we have $x\refpreceq y$ if $x$ is a descendant of the left child and $y$ a descendant of the right child of the lowest common ancestor of $x$ and $y$.
We apply \Cref{bdim-aux-composition} to the tree $T$ with labeling $\rho^\geq$ and linear order $\refpreceq$.
We obtain $d$ linear orders $\preceq_1,\ldots,\preceq_d$ on $X$, where $d\in\Oh(\size{\Lambda}\log{\size{\Lambda}})=\Oh(2^{\cramped{t^2}}t)$, such that for distinct $x,y\in X$, from the Boolean values $[x\refpreceq y],[x\preceq_1y],\ldots,[x\preceq_dy]$ one can uniquely determine the values $\baggb{u}{x}$ and $\baggb{u}{y}$, where $u$ is the lowest common ancestor of $x$ and $y$ in $T$.
Analogously, applying \Cref{bdim-aux-composition} again but this time with labeling $\rho^\leq$, we obtain $d$ linear orders $\preceq'_1,\ldots,\preceq'_d$ on $X$ that let us determine the values $\baglb{u}{x}$ and $\baglb{u}{y}$.

Now, for any distinct $x,y\in X$, it follows from \Cref{fact:convexity} that $x\leq y$ if and only if there is an index $i\in[t]$ such that $(i,1)\in\baggb{u}{x}\cap\baglb{u}{y}$, where $u$ is the lowest common ancestor of $x$ and $y$ in $T$.
Therefore, whether $x\leq y$ can be determined given the Boolean values $[x\refpreceq y],[x\preceq_1y],\ldots,[x\preceq_dy],[x\preceq'_1y],\ldots,[x\preceq'_dy]$.
This constitutes a decoder, so we can set $b=2d+1\in\Oh(2^{\cramped{t^2}}t)$.
\end{proof}

\section{Proof of Colcombet's Theorem}\label{sec:colcombet}

In this section, we provide a proof of Colcombet's Theorem in the following stronger form.

\begin{theorem}[Colcombet's Theorem]\label{thm:colcombet}
For every finite semigroup\/ $(\Lambda,\cdot)$ and for\/ $p=\size{\Lambda}$, there are subsets\/ $E_1,\ldots,E_p$ of\/ $\Lambda$ with the following properties:
\begin{enumerate}
\item\label{item:colcombet-1} For all\/ $h\in[p]$ and\/ $\beta,\beta'\in E_h$, we have\/ $\beta\cdot\beta'=\beta$.
\item\label{item:colcombet-2} Every rooted tree\/ $T$ equipped with a\/ $(\Lambda,\cdot)$-labeling\/ $\lambda$ has a split of order\/ $p$ such that for every\/ $h\in[p]$ and every pair of\/ $h$-neighbors\/ $u\downto v$, we have\/ $\lambda(u,v)\in E_h$.
\end{enumerate}
\end{theorem}

Statements \ref{item:colcombet-1} and \ref{item:colcombet-2} above indeed imply that every rooted tree $T$ equipped with a $(\Lambda,\cdot)$-labeling $\lambda$ has a split of order $\size{\Lambda}$ that is forward Ramseyan with respect to $\lambda$, which is the form of Colcombet's Theorem that we stated in \Cref{sec:prelims}.

For the proof, let $(\Lambda,\cdot)$ be a finite semigroup and $p=\size{\Lambda}$.
Let $\bar\Lambda=\Lambda\cup\set{\iota}$ for some $\iota\notin\Lambda$.
Extend $\cdot$ to $\bar\Lambda$ by setting $\alpha\cdot\iota=\iota\cdot\alpha=\alpha$ for every $\alpha\in\Lambda$ and $\iota\cdot\iota=\iota$, so that $\iota$ is an identity of $(\bar\Lambda,\cdot)$.
We use multiplicative notation for the semigroup operation $\cdot$, skipping the symbol $\cdot$ and writing $\alpha^k$ for the $k$th power of $\alpha$ under $\cdot$ when $k\geq 1$.
Also, we set $\alpha^0=\iota$ by convention, for all $\alpha\in\bar\Lambda$.

For $\alpha,\beta\in\Lambda$, write
\begin{align*}
\alpha&\leq\beta\qquad\text{if and only if}\qquad\text{there exist }\gamma,\gamma'\in\bar\Lambda\text{ such that }\gamma\alpha\gamma'=\beta{,} \\
\alpha&<\beta\qquad\text{if and only if}\qquad\alpha\leq\beta\text{ and }\beta\nleq\alpha{.}
\end{align*}
It follows that $\leq$ is a preorder (i.e., a reflexive and transitive binary relation) and $<$ is a strict partial order on $\Lambda$.
For $\alpha\in\Lambda$, define
\[E(\alpha)=\set{\beta\in\Lambda\colon\alpha\beta=\alpha\text{ and }\alpha\leq\beta}{.}\]

\begin{claim}\label{claim:E}
For all $\alpha\in\Lambda$ and $\beta,\beta'\in E(\alpha)$, we have $\beta\beta'=\beta$.
\end{claim}

\begin{proof}
Let $\beta,\beta'\in E(\alpha)$.
Thus, $\alpha\beta=\alpha$, $\alpha\beta'=\alpha$, and $\alpha\leq\beta$, which means that there exist $\gamma,\gamma'\in\bar\Lambda$ such that $\beta=\gamma\alpha\gamma'$.
This implies that $\beta=(\gamma\alpha)^k(\gamma')^k$ for every $k\geq 1$, by induction: this holds for $k=1$, and assuming it holds for some $k\geq 1$, we have
\begin{align*}
\beta&=(\gamma\alpha)^k(\gamma')^k=(\gamma\alpha)^{k-1}\gamma\alpha(\gamma')^k=(\gamma\alpha)^{k-1}\gamma\alpha\beta(\gamma')^k \\
&=(\gamma\alpha)^k\beta(\gamma')^k=(\gamma\alpha)^k\gamma\alpha\gamma'(\gamma')^k=(\gamma\alpha)^{k+1}(\gamma')^{k+1}{,}
\end{align*}
so it holds for $k+1$.
Since $\size{\Lambda}=p$, there exist integers $k$ and $\ell$ with $1\leq k<\ell\leq p+1$ such that $(\gamma')^k=(\gamma')^\ell$.
By the above, we have
\begin{gather*}
\beta=(\gamma\alpha)^\ell(\gamma')^\ell=(\gamma\alpha)^{\ell-k-1}\gamma\alpha(\gamma\alpha)^k(\gamma')^k=(\gamma\alpha)^{\ell-k-1}\gamma\alpha\beta=(\gamma\alpha)^{\ell-k-1}\gamma\alpha{,} \\
\beta\beta'=(\gamma\alpha)^{\ell-k-1}\gamma\alpha\beta'=(\gamma\alpha)^{\ell-k-1}\gamma\alpha=\beta{.}\qedhere
\end{gather*}
\end{proof}

Consider a rooted tree $T$ equipped with a $(\Lambda,\cdot)$-labeling $\lambda$.
Let $r$ be the root of $T$.
We define an operator $v\mapsto v'$ on the inner nodes of $T$ as follows, by induction.
Let $v$ be an inner node of~$T$.
Having defined $u'$ for every node $u$ with $r\downto u\downto v$, let $v'$ be the lowest (i.e., closest to~$v$) node $u$ of $T$ such that the following conditions hold:
\[u\downto v\qquad\text{and}\qquad\text{if $u\neq r$, then $\lambda(u',u)>\lambda(u,v)$.}\]
Such a node exists, as $r$ satisfies the conditions.
Also, we call $u$ a \emph{valid ancestor} of $v$ in $T$ if just the conditions above hold, but $u$ is not necessarily the lowest node satisfying those conditions.

\begin{claim}\label{claim:trace}
If\/ $v'\downto u\downto v$, then\/ $v'\downtoeq u'$.
\end{claim}

\begin{proof}
The assertion is trivial when $v'=r$, so assume $v'\neq r$.
We have $\lambda(v'',v')>\lambda(v',v)=\lambda(v',u)\lambda(u,v)\geq\lambda(v',u)$.
Therefore, $v'$ is a valid ancestor of $u$ in $T$, so $v'\downtoeq u'$.
\end{proof}

Let $f\colon\Lambda\to[p]$ be a bijection such that for any $\alpha,\beta\in\Lambda$, if $\alpha<\beta$, then $f(\alpha)<f(\beta)$.
Let a split $s\colon U\to[p]$ of $T$ (where $U$ is the set of inner nodes of $T$) be defined by $s(u)=f(\lambda(u',u))$.

\begin{claim}\label{claim:split}
For every $\alpha\in\Lambda$ and every pair $u\downto v$ of $f(\alpha)$-neighbors in $T$, we have $\lambda(u,v)\in E(\alpha)$.
\end{claim}

\begin{proof}
Fix $\alpha\in\Lambda$ and a pair $u\downto v$ of $f(\alpha)$-neighbors in $T$.
If follows that $\lambda(u',u)=\alpha=\lambda(v',v)$.
If $u\downtoeq v'$, then the fact that $\lambda(v'',v')>\lambda(v',v)=\alpha$ implies $s(v')=f(\lambda(v'',v'))>f(\alpha)$, which contradicts the assumption that $u$ and $v$ are $f(\alpha)$-neighbors.
Thus $v'\downto u$, so $v'\downtoeq u'$ by \Cref{claim:trace}.
Suppose $v'\downto u'$.
Then $v'\downtoeq u''$, again by \Cref{claim:trace}.
Now, if $v'\downto u''$, then $\alpha=\lambda(v',v)=\lambda(v',u'')\lambda(u'',u')\lambda(u',v)$, and if $v'=u''$, then $\alpha=\lambda(v',v)=\lambda(u'',u')\lambda(u',v)$.
In either case, $\alpha\geq\lambda(u'',u')>\lambda(u',u)=\alpha$, which is a contradiction.
Thus $v'=u'$, which implies that $\alpha=\lambda(v',v)=\lambda(u',u)\lambda(u,v)=\alpha\lambda(u,v)$.
In particular, $\alpha\geq\lambda(u,v)$.
Furthermore, since $v'\downto u$, the node $u$ is not a valid ancestor of $v$ in $T$, so $\alpha=\lambda(u',u)\ngtr\lambda(u,v)$.
Therefore, we also have $\alpha\leq\lambda(u,v)$.
Since $\alpha=\alpha\lambda(u,v)$ and $\alpha\leq\lambda(u,v)$, we conclude that $\lambda(u,v)\in E(\alpha)$.
\end{proof}

To complete the proof of \Cref{thm:colcombet}, recall the bijection $f\colon\Lambda\to[p]$, and define $E_1,\ldots,E_p$ so that $E_{f(\alpha)}=E(\alpha)$ for every $\alpha\in\Lambda$.
Statements \ref{item:colcombet-1} and \ref{item:colcombet-2} in the assertion of \Cref{thm:colcombet} then follow from \Cref{claim:E,claim:split}, respectively.

\section*{Acknowledgments}
We thank Michał Seweryn for several helpful discussions on \Cref{thm:excluding-a-minor}, and in particular for suggesting the use of Ramsey's Theorem in the proof of \Cref{lem:Kelly_poset_to_minor}.
We are grateful to the anonymous reviewers for their careful reading and helpful comments on an earlier version of the paper.

\bibliographystyle{alpha}
\bibliography{ref}

\end{document}